\documentclass[11pt]{article}
\usepackage{amsmath}
\usepackage{amsthm}
\usepackage{amssymb}
\usepackage{graphics}
\usepackage{bm}
\usepackage{setspace}
\usepackage{authblk}

\usepackage{amsfonts}
\usepackage[cal=bickham,calscaled=1.05]{mathalfa}
\usepackage{eucal}
\usepackage{latexsym}
\usepackage{mathdots}
\usepackage{mathrsfs}
\usepackage{enumitem}

\usepackage{fullpage}
\usepackage{color}

\newcommand{\be}{\begin{equation}}
\newcommand{\ee}{\end{equation}}
\newcommand{\ba}{\begin{eqnarray}}
\newcommand{\ea}{\end{eqnarray}}
\newcommand{\bal}{\begin{align}}
\newcommand{\eal}{\end{align}}
\newcommand{\baln}{\begin{align*}}
\newcommand{\ealn}{\end{align*}}
\newcommand{\bi}{\begin{itemize}}
\newcommand{\ei}{\end{itemize}}
\newcommand{\bn}{\begin{enumerate}}
\newcommand{\en}{\end{enumerate}}
\newcommand{\bbm}{\begin{bmatrix}}
\newcommand{\ebm}{\end{bmatrix}}
\newcommand{\bpm}{\begin{pmatrix}}
\newcommand{\epm}{\end{pmatrix}}
\newcommand{\bsm}{\left ( \begin{smallmatrix}}
\newcommand{\esm}{\end{smallmatrix} \right) }
\newcommand{\bp}{\begin{proof}}
\newcommand{\ep}{\end{proof}}
\newcommand{\nn}{\nonumber}

\newcommand{\mr}{\ensuremath{\mathrm}}
\newcommand{\scr}{\ensuremath{\mathscr}}

\newcommand{\mc}{\ensuremath{\mathcal}}
\newcommand{\mf}{\ensuremath{\mathfrak}}
\newcommand{\ov}{\ensuremath{\overline}}
\newcommand{\sm}{\ensuremath{\setminus}}

\newcommand{\ga}{\ensuremath{\gamma}}

\newcommand{\La}{\ensuremath{\Lambda }}
\newcommand{\la}{\ensuremath{\lambda }}
\newcommand{\om}{\ensuremath{\omega}}

\newcommand{\eps}{\ensuremath{\epsilon }}

\def\cH{\mathcal{H}}
\def\bH{\mathbb{H}}
\def\cJ{\mathcal{J}}
\def\C{\mathbb{C}}

\def\D{\mathbb{D}}

\def\N{\mathbb{N}}
\def\B{\mathbb{B}}

\def\A{\mathbb{A} _d}
\def\Anot{\mathbb{A} _d ^{(0)} }

\def\fp{\mathbb{C} \{ \mathfrak{z} _1 , ..., \mathfrak{z} _d \} }

\def\ncm{ \scr{A} _d   ^\dag }
\def\posncm{ \left( \scr{A} _d   \right) ^\dag _+ }
\def\hardy{\mathbb{H} ^2 _d}
\def\mult{\mathbb{H} ^\infty _d}
\def\nbran{\mr{Ran} \, }
 
\def\nbre{\mr{Re} \, }
\def\nbim{\mr{Im} \, }
\def\mrt{\mathrm{t}}
\def\fu{\mathfrak{u}}
\def\ft{\mathfrak{t}}
\def\fw{\mathfrak{w}}

\newcommand{\F}{\ensuremath{\mathbb{F} }}

\newcommand{\ip}[2]{\ensuremath{\langle {#1} , {#2} \rangle}}

\numberwithin{equation}{section}
\numberwithin{subsection}{section}

\newtheorem{thm}{Theorem}[section]
\newtheorem{claim}{Claim}

\newtheorem{lemma}[thm]{Lemma}
\newtheorem{prop}[thm]{Proposition}
\newtheorem{cor}[thm]{Corollary}
\newtheorem*{thm*}{Theorem}

\theoremstyle{definition}
\newtheorem{defn}[thm]{Definition}
\newtheorem{remark}[thm]{Remark}
\newtheorem{eg}[thm]{Example}

\newcommand{\Addresses}{{
  \bigskip
  \footnotesize
  
  Michael T.~Jury, \textsc{Department of Mathematics, University of Florida}\par\nopagebreak
  \textit{E-mail address:} \texttt{mjury@ad.ufl.edu}

\medskip

  Robert T.~W.~Martin, \textsc{Department of Mathematics, University of Manitoba}\par\nopagebreak
  \textit{E-mail address:} \texttt{Robert.Martin@umanitoba.ca}
    
\medskip

Edward J. Timko, \textsc{Department of Mathematics, Georgia Institute of Technology}\par\nopagebreak
\textit{E-mail address:} \texttt{etimko6@gatech.edu }

\vspace{1cm}

}}

\title{A non-commutative F\&M Riesz Theorem}

\author{Michael T. Jury\thanks{Supported by NSF grant DMS-1900364}}
\affil[1]{\footnotesize University of Florida}

\author[2]{Robert T.W. Martin\thanks{Supported by NSERC grant 2020-05683}}
\affil[2]{\footnotesize University of Manitoba}

\author[3]{Edward J. Timko\thanks{Partially supported by a PIMS postdoctoral fellowship}}
\affil[3]{\footnotesize Georgia Institute of Technology}

\date{}

\begin{document}

\bibliographystyle{plain}
\maketitle

\begin{abstract}
	We extend results on analytic complex measures on the complex unit circle to a non-commutative multivariate setting.
	Identifying continuous linear functionals on a certain self-adjoint subspace of the Cuntz--Toeplitz $C^*-$algebra, the \emph{free disk operator system}, with non-commutative (NC) analogues of complex measures,
	we refine a previously developed Lebesgue decomposition for positive NC measures to establish an NC version of the Frigyes and Marcel Riesz Theorem for `analytic' measures, \emph{i.e.} complex measures with vanishing positive moments. The proof relies on novel results on the order properties of positive NC measures that we develop and extend from classical measure theory. 
\end{abstract}

\section{Introduction}

The Riesz--Markov theorem identifies any finite and regular Borel measure on the complex unit circle with a bounded linear functional on the Banach space of continuous functions. By the Weierstrass approximation theorem, the Banach space of continuous functions on the circle is the supremum norm--closure of the linear span of the disk algebra and its conjugate algebra. The disk algebra is often defined as the unital Banach algebra of analytic functions in the complex unit disk with continuous boundary values. This is completely isometrically isomorphic to the unital norm--closed operator algebra generated by the shift operator, $S=M_z$, of multiplication by the independent variable $z$ on $H^2$. Here, $H^2$ denotes the \emph{Hardy space}, the Hilbert space of analytic functions in the complex unit disk that have square--summable Taylor series coefficients at $0$, equipped with the $\ell ^2-$inner product of these coefficients.

An immediate non-commutative (NC) multivariate generalization of $H^2$ is then $\hardy$, the \emph{NC Hardy space} or \emph{full Fock space}, which consists of square--summable power series in several formal NC variables $\mf{z} = (\mf{z} _1 , \cdots , \mf{z} _d )$. Elements of $\hardy$ are power series,
$$ h (\mf{z} ) = \sum _{\om \in \F ^d} \hat{h} _\om \mf{z} ^\om, $$ with square--summable coefficients $\hat{h} _\om \in \C$. The \emph{free monoid}, $\F ^d$, is the set of all words, $\om = i_1 \cdots i_n$, $1 \leq i_k \leq d$, in the $d$ letters $\{ 1 , \cdots , d \}$. This is a monoid with product given by concatenation of words and the unit $\varnothing$ is the empty word containing no letters. Given any $\om =i_1 \cdots i_n \in \F ^d$, the \emph{free monomial} $\mf{z} ^\om$ is $\mf{z} _{i_1} \cdots \mf{z} _{i_n}$ and $\mf{z} ^\varnothing = 1$, viewed as a constant NC function. As in the classical setting, left multiplication $L_k := M^L _{\mf{z} _k}$ by any of the $d$ independent NC variables defines an isometry on $\hardy$, and we call these isometries the \emph{left free shifts}. These play the role of the shift in this NC Hardy space theory.

The NC analogues of the disk algebra, the continuous functions (equivalently, the \emph{disk operator system}, the supremum norm--closed linear span of the disk algebra and its conjugates) and positive measures are then the \emph{free disk algebra},
$\A := \mr{Alg} \{ I , L_1 , \cdots , L_d \} ^{-\| \cdot \|}$, the \emph{free disk system},
$$ \scr{A} _d := \left( \A + \A ^* \right) ^{-\| \cdot \|}, $$  and \emph{NC measures}, \emph{i.e.} bounded linear functionals on the free disk system. We denote the Banach space of all NC measures by $\ncm$ and the cone of positive NC measures by $\posncm$. In \cite{JM-ncFatou,JM-ncld}, the first two named authors constructed the Lebesgue decomposition of any positive NC measure with respect to a canonical NC Lebesgue measure and showed that the sets of absolutely continuous and singular NC measures are positive and hereditary cones, in syzygy with classical measure theory. 

The isometric left free shifts $L_k = M^L _{\mf{z} _k}$ on $\hardy$ have pairwise orthogonal ranges and it follows that the linear map $L := \left( L_1 , \cdots , L_d \right) : \hardy \otimes \C ^d \rightarrow \hardy$ is an isometry from several copies of $\hardy$ into itself. Such an isometry is called a \emph{row isometry}. By the classical Wold decomposition theorem, any isometry on Hilbert space decomposes as the direct sum of a \emph{pure isometry}, \emph{i.e.} an isometry unitarily equivalent to copies of the shift on $H^2$, and a unitary operator. There is an exact analogue of the Wold decomposition theorem for row isometries, established by G. Popescu; any row isometry decomposes as the direct sum of a pure row isometry, unitarily equivalent to several copies of the left free shift and a surjective or \emph{Cuntz} row isometry, the multivariate analogue of a unitary operator \cite{Pop-dil}. The $C^*$-algebra $\scr{E} _d=C^* \{ I , L_1 , \cdots , L_d \}$ and its quotient, $\scr{O} _d$, by the compact operators are the celebrated \emph{Cuntz--Toeplitz} and \emph{Cuntz algebras} \cite{Cuntz}, respectively. These are universal $C^*$-algebras for row isometries and are important objects in $C^*$-algebra theory. Any $*-$representation, $\pi$, of $\scr{E}_d$ determines and is uniquely determined by the row isometry $\Pi := \pi (L)=(\pi(L_1),\cdots,\pi(L_d))$.  

The Lebesgue--von Neumann--Wold decomposition of a single isometry on Hilbert space further splits the unitary direct summand into the direct sum of a unitary with absolutely continuous spectral measure (with respect to Lebesgue measure) and a singular unitary. In \cite{MK-rowiso}, M. Kennedy extended the Lebesgue--von Neumann--Wold decomposition for single isometries to row isometries. A new feature in this decomposition is that the Cuntz direct summand of any row isometry generally splits as the direct sum of three different types which we will call \emph{absolutely continuous (AC) Cuntz}, \emph{von Neumann type} (called \emph{singular} in \cite{MK-rowiso}) and \emph{dilation--type}. 

A Gelfand--Naimark--Segal (GNS) construction applied to the free disk algebra and any positive NC measure $\mu$ produces a Hilbert space $\hardy (\mu )$ and a row isometry $\Pi _\mu$ acting on $\hardy (\mu )$. Any cyclic row isometry (or $*-$representation of $\scr{E} _d$) can be obtained as the GNS row isometry of a positive NC measure, as recorded in Lemma \ref{cyclicCuntz}. The role of normalized Lebesgue measure on the circle is played by the so-called vacuum state of the Fock space, $m (L^\om ) = \ip{1}{L^\om 1} _{\bH ^2}$ where $1$ is identified with the \emph{vacuum vector} $\mf{z} ^\varnothing$. We call $m$ \emph{NC Lebesgue measure}. Observe that if we identify normalized Lebesgue measure, $m$, with a positive linear functional on the norm-closure of the linear span of the disk algebra and its conjugate algebra, via the Riesz--Markov Theorem, then $m(S^k) = \ip{1}{S^k 1}_{H^2}$. Thus, our definition of NC Lebesgue measure recovers normalized Lebesgue measure when $d=1$. In \cite{JM-ncFatou,JM-ncld}, the first two named authors constructed the Lebesgue decomposition of any positive NC measure, $\mu \in \posncm$, $\mu = \mu _{ac} + \mu _s$, where $\mu _{ac}, \mu _s \in \posncm$ are absolutely continuous and singular, respectively, with respect to NC Lebesgue measure in the sense of \cite[Corollary 8.12, Corollary 8.13]{JM-ncld}. In particular, it is shown that the sets of absolutely continuous and singular NC measures are positive cones that are \emph{hereditary} in the sense that if $\mu , \la \in \posncm$ and $\la$ is absolutely continuous or singular, then $\mu \leq \la$ implies that $\mu$ is also absolutely continuous or singular, respectively, in parallel with classical measure theory. In our NC Lebesgue decomposition, $\mu$ is absolutely continuous if and only if its GNS row isometry $\Pi_\mu$ is the direct sum of pure and AC Cuntz row isometries, while $\mu$ is singular if and only if $\Pi_\mu$ is the direct sum of von Neumann--type and dilation--type row isometries \cite[Corollary 8.12, Corollary 8.13]{JM-ncld}.

In this paper we further refine the NC Lebesgue decomposition of \cite{JM-ncld} by proving that the sets of dilation--type and von Neumann type NC measures (NC measures whose GNS row isometries are dilation or von Neumann type) are both positive hereditary cones in Theorem \ref{dilvNcone}. This yields, in Corollary \ref{refncld}, a refined Kennedy--Lebesgue--von Neumann decomposition of any positive NC measure. We then apply this decomposition to obtain analogues of classical results due to Frigyes and Marcel Riesz characterizing analytic (complex) NC measures, which are those bounded linear functionals on the free disk system which annihilate $\{L^\beta 1 | \ \beta\in\mathbb{F}^d\backslash\{\varnothing\}\}$. In particular, in Theorem \ref{BabyFnM}, we show that a complex NC measure is analytic if and only if its absolutely continuous, singular, dilation and von Neumann parts are each also analytic. We apply this to establish an analogue of the classical F\&M Riesz theorem in Theorem \ref{ncFnM}. The classical theorem states that analytic measures are absolutely continuous \cite{Riesz}, \cite[Chapter 4]{Hoff}. We remark that the NC F\&M Riesz theorem obtained here is related, but not equivalent to, one previously developed in \cite[Theorem A]{CMT-analytic} using different techniques. Both Theorem \ref{ncFnM} of the present paper and Theorem A of \cite{CMT-analytic} conclude that an analytic NC measure need not be absolutely continuous, however the results of this paper and those of \cite{CMT-analytic} describe the obstruction in different ways. We discuss the relationship between these two results in Remark \ref{relate}.

\section{Background and Notation}

Throughout, given $h,h'$ in a Hilbert space $\mathcal{H}$, we denote the inner product of $h$ and $h'$ by $\langle h,h'\rangle$, with $\langle \cdot ,\cdot\rangle$ being conjugate linear in the first argument and linear in the second.
We borrow notation from \cite{JM-ncFatou,JM-ncld}. The \emph{Fock space} or \emph{NC Hardy space} is
$$ \hardy := \left\{ \left. h (\mf{z} ) = \sum _{\om \in \F ^d } \hat{h} _\om \mf{z} ^\om \right| \ \hat{h} _\om \in \C, \ \sum_{\om \in \F^d} | \hat{h} _\om | ^2 < + \infty \right\}, $$ equipped with the $\ell ^2-$inner product of the power series coefficients $\hat{h} _\om$.
Elements of the NC Hardy space can be viewed as free non-commutative functions in the \emph{NC unit row--ball}, $\B ^d _\N$, of all finite--dimensional strict row contractions:
$$ \B ^d _\N = \bigsqcup _{n=1} ^\infty \B ^d _n; \quad \B ^d _n := \left\{\left. Z = (Z _1 , \cdots, Z_d ) \in \C ^{n \times n} \otimes \C ^{1\times d}  \right| \ Z Z^* = Z_1 Z_1 ^* + \cdots Z_d Z_d ^* < I_n \right\}; $$ see, for example, \cite{JMS-ncBSO,JM-ncFatou,JM-ncld,Pop-freeholo,SSS,BMV} for details. 

The \emph{left free shift} $L = ( L_1 , \cdots , L_d )$ is the row isometry on the full Fock space $\mathbb{H}^2 _d$ whose component operators act by left multiplication by the independent variables, $L_k = M^L _{\mf{z} _k}$.
The free or NC disk algebra is $\A := \mr{Alg} \{ I, L_1 , \cdots , L _d \} ^{-\| \cdot \|} $, the \emph{free disk system} is $\scr{A} _d := \left( \A + \A ^* \right) ^{-\| \cdot \| }$. This is a self-adjoint unital norm-closed subspace of operators, \emph{i.e.} an operator system. A (complex) NC measure is a bounded linear functional on the free disk system. The set of all complex NC measures is denoted by $\ncm$, and the positive NC measures by $\posncm$.
We remark that any $\mu\in\posncm$ is uniquely determined by the moments $(\mu(L^\alpha))_{\alpha\in\mathbb{F}^d}$.

Given $a \in \A$, we write $a= a(L) = M^L_{a}$ for the operator of left multiplication by $a(\mf{z})$ on $\hardy$, where 
$$ a(\mf{z} ) = \sum _{\om \in \F ^d} \hat{a} _\om \mf{z} ^\om, $$
is the NC function determined by $a$.
The partial Ces\`aro sums of the series for $a(L)$ converge in the strong operator topology to $a (L) = M^L_{a(\mf{z})}$, as shown in \cite[Lemma 1.1]{DP-inv}.
More generally, given a row isometry $\Pi$, we write $a\mapsto a(\Pi)$ for the unique representation of $\A$ satisfying $\mathfrak{z}^\alpha(\Pi)=\Pi^\alpha$ for all $\alpha\in\F ^d$.

The free disk system has the \emph{semi-Dirichlet property} \cite{DK-dilation}:
$$ \A ^* \A \subseteq (\A + \A ^* ) ^{-\| \cdot \|} = \scr{A} _d. $$
The semi-Dirichlet property enables one to apply a Gelfand--Naimark--Segal (GNS)-type construction to $(\mu , \A )$, where $\mu \in \posncm$ is any positive NC measure.
One obtains a GNS--Hilbert space $\mathbb{H} ^2 _d (\mu )$ as the norm closure of the free disk algebra (modulo vectors of zero length) with respect to the pre-inner product
$$ (a_1,a_2)\mapsto  \mu \left( a_1 ^* a_2\right). $$ 
Elements of this Hilbert space are equivalence classes $a + N_\mu$, $a \in \A$, where $N_\mu$ is the left ideal of all those $a \in \A$ such that $\mu(a ^* a )=0$.
This construction provides a representation $\pi _\mu : \A \rightarrow \mc{L} (\mathbb{H} ^2 _d )$, where
$$ \pi _\mu (a) \,  a' + N_\mu  := a a' + N_\mu. $$ 
This is a unital completely isometric isomorphism, so that the image of the left free shifts,
$$ \Pi _\mu := \left( \Pi _{\mu; 1}, \cdots , \Pi _{\mu ; d} \right), \quad \quad \Pi _{\mu ; k} := \pi _\mu (L_k), \quad k=1,\ldots,d, $$
defines a row isometry, which we call the \emph{GNS row isometry of $\mu$}, acting on the GNS space $\mathbb{H} ^2 _d (\mu )$. The original positive NC measure $\mu \in \posncm$ then has the spacial representation
$$ \mu (L ^\om ) = \ip{ I + N_\mu}{\Pi _\mu ^\om \, I + N_\mu }. $$
If $\mu, \la \in \posncm$ and $\mu \leq \la$, then the map
$$ p + N_\la \,  \mapsto \, p + N_\mu, $$ extends by continuity to a contraction $E_{\mu, \la}:\hardy (\la )\rightarrow\hardy (\mu )$ with dense range. In this case, setting $D _{\mu , \la} := E^* _{\mu, \la} E_{\mu, \la}$, we have 
$$ \mu (L ^\om ) = \ip{ I + N _\la }{D _{\mu , \la} \, \Pi _\la ^\om (I + N_\la)}, $$ and $D \geq 0$ can be viewed as the `NC Radon--Nikodym derivative' of $\mu$ with respect to $\la$. 
\begin{remark}
The NC Radon--Nikodym derivative, $D_{\mu , \la}$ is $\la-$ or $\Pi _\la-$\emph{Toeplitz} in the sense that
$$ \Pi _{\la ; j} ^* D_{\mu , \la} \Pi _{\la ; k} = \delta _{j,k} D_{\mu , \la}. $$ Here, recall that a bounded linear operator $T$ on the Hardy space $H^2$ is called \emph{Toeplitz} if $T$ is the compression $T_f := P_{H^2} M_f | _{H^2}$ of the bounded multiplication operator $M_f$ for some $f \in L^\infty (\partial \D)$. A theorem of Brown and Halmos identifies the Toeplitz operators as the set of all bounded linear operators, $T \in \scr{L} (H ^2 )$, obeying the simple algebraic condition
$$ S^* T S = T, $$ where $S=M_z$ is the shift \cite[Theorem 6]{BH-Toep}. 
\end{remark}
We refer to $E_{\mu, \la}$ as the \textit{co-embedding} determined by the inequality $\mu\leq\lambda$, since its adjoint is injective.
Given $\mu,\nu,\lambda\in\posncm$ satisfying $\mu\leq\nu\leq\lambda$, it follows that
\[ E_{\mu,\nu}E_{\nu,\lambda}=E_{\mu, \la}. \]
We remark that that $E_{\mu, \la}^*$ is unitarily equivalent to an embedding of NC reproducing kernel Hilbert spaces; see \cite[Lemma 3]{JM-ncFatou}\cite{JM-ncld}.

We now record the fact that any cyclic row isometry is unitarily equivalent to the GNS row isometry of a positive NC measure.
The proof is straightforward and thus omitted.
\begin{lemma} \label{cyclicCuntz}
	Let $\Pi$ be a cyclic row isometry on a Hilbert space $\cH$ with a cyclic vector $x$. Define a positive NC measure $\mu \in \posncm$ by setting $\mu (L^\om) = \ip{x}{\Pi ^\om x}$. The map $U_x : \hardy (\mu ) \rightarrow \cH$ defined by $U_x \, a + N_\mu = a(\Pi ) x$, $a \in \A$, extends to a surjective isometry that intertwines $\Pi_\mu$ and $\Pi$. 
\end{lemma}

Let $\mathfrak{L}^\infty _d := \mr{Alg} \{ I, L_1 , \cdots , L_d \} ^{-\text{weak-}*}$ denote the \emph{left free analytic Toeplitz algebra} or the \emph{Free Hardy Algebra}. From a result of Davidson--Pitts \cite[Corollary 2.12]{DP-inv}, it follows that $\mathfrak{L}^\infty _d = \mr{Alg} \{ I, L_1 , \cdots , L_d \} ^{-\text{WOT}}$, the closure of $\mr{Alg} \{ I, L_1 , \cdots , L_d \}$ in the weak operator topology (WOT). That is, $\mathfrak{L}^\infty _d$ is a \emph{free semigroup algebra}, the unital WOT--closed operator algebra generated by a row isometry \cite{KRD-semi}. The algebra $\mathfrak{L}^\infty_d$ can also be identified with the left multiplier algebra of $\hardy$, viewed as a non-commutative reproducing kernel Hilbert space (RKHS) \cite{SSS,JMS-ncBSO,BMV}. 
We remark that this left multiplier algebra is equal to the unital Banach algebra $\mult$ of all free NC functions in the NC unit row-ball $\B^d _\N$ that are uniformly bounded in supremum norm \cite{SSS,Pop-freeholo}. 
The NC or \emph{free Toeplitz system} is 
$$ \scr{T} _d := \left( \mathfrak{L}^\infty _d + (\mathfrak{L}^\infty _d ) ^* \right) ^{-\text{weak-}*} = \scr{A} _d ^{-\text{weak-}*}. $$
We also use the right free shift $R=(R_1, \cdots , R_d )$, the row isometry of right multiplications $R_k := M^R _{\mf{z} _k}$ by the independent NC variables on the Fock space. The \emph{right free analtyic Toeplitz algebra} is $\mathfrak{R}^\infty _d := \mr{Alg} \{ I,R_1 , \cdots , R_d \} ^{-\mr{WOT}}$.

\subsection{Stucture of GNS row isometries}

By \cite{MK-rowiso}, any row isometry $\Pi$ on a Hilbert space $\cH$ can be decomposed as the direct sum of four types of row isometries:
$$ \Pi = \Pi _L \oplus \Pi _{\mr{ACC}} \oplus \Pi _{\mr{dil}} \oplus \Pi _{\mr{vN}}. $$ Here $\Pi _L$ is \emph{pure type$-L$} if it is unitarily equivalent to an ampliation of $L$. The remaining three types are Cuntz, \emph{i.e.} surjective row isometries. A Cuntz row isometry is the multi-variable analogue of a unitary in our context and we sometimes call a Cuntz row isometry a Cuntz unitary. We also call any pure type$-L$ row isometry a \emph{pure row isometry}. The summand $\Pi _{\mr{ACC}}$ is \emph{absolutely continuous Cuntz} (or \emph{AC Cuntz} or \emph{ACC}), meaning that $\Pi _{\mr{ACC}}$ is a Cuntz row isometry and the free semigroup algebra it generates is completely isometrically isomorphic and weak$-*$ homeomorphic to $\mathfrak{L}^\infty _d$. The summand $\Pi _{\mr{vN}}$ is \emph{von Neumann type}, meaning the free semigroup algebra it generates is self-adjoint and hence a von Neumann algebra. The leftover piece $\Pi _{\mr{dil}}$ is of \emph{dilation type}. A row isometry $\Pi$ is of dilation type if it has no direct summand of the previous three types. Any dilation type row isometry $\Pi$ has a block upper triangular decomposition
	$$ \Pi \simeq \bpm L \otimes I & * \\ 0 & T \epm, $$
so that $\Pi$ has a restriction to an invariant subspace that is unitarily equivalent to a pure row isometry and $\Pi$ is the minimal row isometric dilation of its compression $T$ to the orthogonal complement of this invariant space. Since $\Pi$ is of Cuntz type, $T$ is necessarily a non-isometric row co-isometry \cite{Pop-dil}.
A row isometry containing only pure type$-L$ and ACC summands is said to be \textit{absolutely continuous} (or \textit{AC}), and a row isometry containing only dilation--type and von Neumann--type summands is said to be \textit{singular}.

We form the set of labels
\[ \mathrm{Types}=\{\text{L},\text{Cuntz},\text{ac},\text{s},\text{ACC},\text{dil},\text{vN}, \text{all}  \}. \]
For a given row isometry $\Pi$, if we write ``$\Pi$ is type $\ft$'', then we mean ``$\Pi$ is pure type $L$'' when $\ft=\mr{L}$, ``$\Pi$ is Cuntz--type'' when $\ft=\mr{Cuntz}$, ``$\Pi$ is absolutely continuous'' when $\ft=\mr{ac}$, ``$\Pi$ is singular'' when $\ft=\mr{s}$, ``$\Pi$ is absolutely continuous Cuntz'' when $\ft=\mr{ACC}$, ``$\Pi$ is dilation--type'' when $\ft=\mr{dil}$, and ``$\Pi$ is von Neumann--type'' when $\ft=\mr{vN}$. We include the trivial type, $\ft = \mr{all}$. If $\Pi$ is of type $\mr{all}$ this simply means that $\Pi$ can be any row isometry.

\begin{defn}\label{D:Types}
	Let $\ft\in\mathrm{Types}$.
	A positive NC measure $\mu \in \posncm$ is said to be \textit{type $\ft$} if its GNS row isometry $\Pi _\mu$ is type $\ft$. 
\end{defn}
Let $\ft\in\mr{Types}$ and consider a row isometry $\Pi$.
There is an orthogonal projection $P_{\ft}$ that commutes with $\Pi$ such that $\Pi$ restricted to the range of $P_{\ft}$ is the type $\ft$ summand of $\Pi$.
In the case of a GNS row isometry $\Pi_{\mu}$, we write $P_{\mu;\ft}$.
Given a positive NC measure $\mu$, we denote by $\mu_{\ft}$ the positive NC measure satisfying
\[ \mu_{\ft}(L^\beta)=\ip{I + N_{\mu}}{P_{\mu;\ft}\Pi^\beta_\mu(I + N_\mu)}, \quad \beta\in\mathbb{F}^d. \]
One may readily verify that $E_{\mu_{\ft}, \mu}$ is a co-isometry satisfying
\[ E_{\mu_{\ft}, \mu}^*E_{\mu_{\ft}, \mu}=P_{\mu;\ft}. \]
Note that $E_{\mu_{\ft}, \mu}^*$ satisfies
\[ E_{\mu_{\ft}, \mu}^*(a+N_{\mu_{\ft}})=P_{\mu;\ft}(a+N_\mu). \]
Because $P_{\mu;\ft}$ is reducing for $\Pi_\mu$, it follows that $E_{\mu_{\ft}, \mu}^*\Pi_{\mu_{\ft}}^\beta=\Pi_{\mu}^\beta E_{\mu_{\ft}, \mu}^*$ for all words $\beta$.
From this, we see that $\Pi_{\mu_\ft}$ is unitarily equivalent to the restriction of $\Pi$ to $\nbran P_{\mu;\ft}$.
Therefore, the GNS row isometry of $\mu_{\ft}$ is type $\ft$, and thus $\mu_{\ft}$ is type $\ft$.

There is an additional projection associated with any row isometry $\Pi$,
and that is the free semigroup algebra structure projection $Q$ of $\Pi$. With $\mf{S} (\Pi ):= \mr{Alg} \{ I , \Pi _1 , \cdots , \Pi _d \} ^{-\text{WOT}}$ denoting the free semigroup algebra of $\Pi$, we denote by $Q$ largest projection in $\mf{S} (\Pi )$ so that $Q \mf{S} (\Pi ) Q$ is self-adjoint \cite[Structure Theorem 2.6]{DKP-structure}.
It has the following properties.
First, $\mf{S} (\Pi )$ has the decomposition 
$$ \mf{S} (\Pi ) = \mr{vN} (\Pi ) Q + Q^\perp \mf{S} (\Pi ) Q^\perp, $$ where $\mr{vN} (\Pi )$ denotes the von Neumann algebra generated by $\{ \Pi _1 , \cdots , \Pi _d \}$.
When $Q\neq I$,
$$ Q^\perp \mf{S} (\Pi ) Q^\perp = \mf{S} (\Pi ) Q^\perp $$ is completely isometrically isomorphic and weak$-*$ homeomorphic to $\mathfrak{L}^\infty _d$.
Here and elsewhere, $P^\perp=I-P$ whenever $P$ is an orthogonal projection.

This is related to the subspace of all \emph{weak$-*$ continuous vectors} for a row isometry $\Pi$. A vector $x \in \cH$ is \textit{weak$-*$ continuous} if the linear functional $ \ell _x  \in \posncm$, defined by $\ell _x ( L^\alpha ) := \ip{x}{\Pi ^\alpha x}$, is weak$-*$ continuous \cite{DLP-ncld}. A bounded operator $X : \bH ^2 _d \rightarrow \cH$ is an \emph{intertwiner for $\Pi$} if 
$$ X L^\alpha = \Pi ^\alpha X, \quad \alpha\in \F ^d. $$
The following theorem combines results of Davidson--Li--Pitts and Kennedy to characterize the set $\mathrm{WC}(\Pi)$ of all weak$-*$ continuous vectors of $\Pi$ in terms of bounded intertwiners.

\begin{thm}[Davidson--Li--Pitts, Kennedy]
	\label{T:DLP}
	Let $\Pi$ be a row isometry on $\cH$.
	\bn
		\item If $x, y \in \mr{WC} (\Pi)$ then the linear functional $\ell _{x,y} : \A \rightarrow \C$, $$ \ell _{x,y} (L^\alpha ) = \ip{x}{\Pi^\alpha y}_\cH,\quad \alpha\in\F ^d, $$ is weak$-*$ continuous.
		\item $\mr{WC}(\Pi)$ is a closed $\Pi -$invariant subspace, and 
			$$ \mr{WC} (\Pi ) = \{Xh \, | \, h\in \mathbb{H}^2_d, \, X\text{ \rm{an intertwiner}}\} $$ 
		\item If $Q$ is the structure projection of $\Pi$, then
			$$ \mr{WC} (\Pi ) = \nbran Q ^\perp. $$ 
	\en
\end{thm}

\begin{proof}
	Items 1 and 2 are directly from \cite[Theorem 2.7]{DLP-ncld}.
	For item 3, we note the following.
	The second dual $\mathbb{A}_d^{\dagger\dagger}$ of $\mathbb{A}_d$ is a free semi-group algebra, and thus there exists a structure projection $\mathfrak{q}$ for $\mathbb{A}_d^{\dagger\dagger}$.
	Let $\hat{\pi}$ denote the weak$-*$ continuous representation of $\mathscr{E}_d^{\dagger\dagger}$ determined by $\pi$.
	By \cite[Proposition 5.2]{DLP-ncld}, $\hat{\pi}(\mf{q})^\bot=Q_{\mr{WC}}$, where $Q_{\mr{WC}}$ denotes the projection onto the closed subspace $WC (\Pi)$.
	In comments following \cite[Proposition 5.2]{DLP-ncld}, it is shown that $\hat{\pi}(\mathfrak{q})=Q$ if and only if $\pi$ is ``regular'', meaning that the $a\mapsto \pi(a)|_{\mr{WC}(\Pi)}$ and $a\mapsto\pi(a)|_{\nbran Q^\bot}$ coincide.
	By \cite[Theorem 3.4]{DLP-ncld} and \cite[Corollary 4.17]{MK-rowiso}, we see that $\pi$ is always regular.
\end{proof}

\section{Convex and order structure of NC measures}

If $0 \leq \mu \leq \la$ are positive NC measures, it is natural to ask whether the contractive co-embedding $E_{\mu, \la} : \mathbb{H} ^2 _d (\la ) \to \mathbb{H} ^2 _d (\mu )$ intertwines the various structure projections of $\mu$ and $\la$.
That is, do we generally have that $E_{\mu, \la}P _{\la ; \ft} = P _{\mu ; \ft} E_{\mu, \la}$, where $\ft \in \mr{Types}$? 
By \cite[Corollary 8.11]{JM-ncld} the sets of absolutely continuous (AC) and singular positive NC measures are positive hereditary cones.
It is therefore also natural to ask whether the sets of von Neumann type and dilation type NC measures are also positive hereditary cones, as we prove in Theorem \ref{dilvNcone} at the end of this section.

\begin{defn}
Let $\mu,\lambda\in(\mathscr{A}_d)^{\dagger}_+$.
We say that $\ft\in\mr{Types}$ is a \textit{hereditary} type if $\mu\leq\lambda$ and $\lambda$ being type $\ft$ together imply that $\mu$ is type $\ft$. A positive sub-cone, $\scr{P}_0 \subseteq \scr{P}$, of a postive cone, $\scr{P}$, is \emph{hereditary}, if $p_0 \in \scr{P}_0$, $p \in \scr{P}$ and $p\leq p_0$ imply that $p \in \scr{P} _0$.
We say that $\ft$ \textit{determines a hereditary cone} if the set of type--$\ft$ positive NC measures form a hereditary cone.
\end{defn}

\begin{lemma}
	Let $\lambda,\mu\in\posncm$ with $\mu\leq\lambda$.
	If $c\in\mathscr{A}_d$ is positive semi-definite, then
	\[ E_{\mu, \la}^*\pi_\mu(c)E_{\mu, \la}\leq \pi_\lambda(c). \]
	\label{L:PosSemi}
\end{lemma}
\begin{proof}
By \cite[Lemma 4.6]{JM-freeAC} the cone of `positive finite sums of squares' of free polynomials, \emph{i.e.} elements of the form
$$ \sum _{j=1} ^N p_j (L) ^* p_j (L), \quad \quad p_j \in \fp, $$ is norm-dense in the cone of positive elements of the free disk system, $\scr{A} _d$. Hence, to prove the claim, it suffices to show that 
$$ E^* _{\mu , \la} \pi _\mu \left( p(L) ^* p(L) \right) E_{\mu , \la} \leq \pi _\la \left( p(L) ^* p(L) \right), $$ for any $p \in \fp$. This is easily verified:
\ba E_{\mu , \la} ^* \pi _\mu \left( p(L) ^* p(L) \right) E_{\mu , \la} & =& E_{\mu , \la} ^* \pi _\mu (p) ^* \pi _\mu (p) E_{\mu ,\la} \nn \\
& = & \pi _\la (p) ^* E_{\mu , \la} ^* E_{\mu , \la}  \pi _\la (p) \nn \\
& \leq & \pi _\la (p) ^* \pi _\la (p) = \pi _\la (p(L) ^* p(L) ). \nn \ea
\end{proof}

\begin{prop}
	Let $\ft\in\mathrm{Types}$.
	\begin{enumerate}
		\item[(a)] If $\mu ,\la \in \posncm$ are such that $E_{\mu, \la}P_{\lambda;\ft}=P_{\mu;\ft}E_{\mu, \la}P_{\lambda;\ft}$ and $\la$ is of type $\ft$, then $\mu$ is also of type $\ft$. In particular, if this formula holds for all $\mu,\lambda\in(\mathscr{A}_d)^{\dagger}_+$ such that $\mu\leq\lambda$, then $\ft$ is a hereditary type.
		\item[(b)] If $E_{\mu, \la}P_{\lambda;\ft}=P_{\mu;\ft}E_{\mu, \la}$ whenever $\mu,\lambda\in(\mathscr{A}_d)^{\dagger}_+$ are such that $\mu\leq\lambda$, then $\ft$ determines a hereditary cone.
			
		\item[(c)] Suppose that $\ft, \fu$ are types and $\la , \mu \in \posncm$ are such that $P_{\la;\ft}^\bot=P_{\la;\fu}$ and similarly for $\mu$. If  $\mu_\ft\leq\lambda_\ft$ and $\mu_\fu\leq \lambda_\fu$, then $E_{\mu, \la}P_{\lambda;\ft}=P_{\mu;\ft}E_{\mu, \la}$.

		\item[(d)] Suppose that $\mu, \la \in \posncm$, $\mu \leq \la$, $\ft$ is a type and $P_{\mu;\ft}E_{\mu, \la}P_{\lambda;\ft}=P_{\mu;\ft}E_{\mu, \la}$. If $\mu$ is of type $\ft$ then $\mu \leq \la _\ft$. 
	\end{enumerate}
	\label{L:EProperties}
\end{prop}
\begin{proof}
	(a) Suppose $\mu\leq\lambda$ and $\lambda$ is type $\ft$.
	Then $P_{\lambda;\ft}=I$ and thus $E_{\mu, \la}=P_{\mu;\ft}E_{\mu, \la}$, and therefore
	\[ \mu(L^\beta)=\ip{E_{\mu, \la}(I + N_\lambda)}{\Pi^\beta_{\mu}E_{\mu, \la}(I + N_{\lambda})}=\ip{E_{\mu, \la}(I + N_\lambda)}{P_{\mu;\ft}\Pi_\mu^\beta E_{\mu, \la}(I + N_\lambda)}=\mu_{\ft}(L^\beta) \]
	for each $\beta\in\mathbb{F}^d$.
	Thus, $\mu$ is type $\ft$.

	(b) The hereditary property follows from (a).
	To see that $\ft$ determines a cone, suppose $\mu,\nu$ are type $\ft$.
	Clearly, $\mu,\nu\leq \mu+\nu$.
	Then
	\[ I=E_{\mu,\mu+\nu}^*E_{\mu,\mu+\nu}+E_{\nu,\mu+\nu}^*E_{\nu,\mu+\nu}, \]
	$P_{\mu;\ft}=I$ and $P_{\nu;\ft}=I$.
	Thus,
	\begin{align*}
		P_{\mu+\nu;\ft} &=(E_{\mu,\mu+\nu}^*E_{\mu,\mu+\nu}+E_{\nu,\mu+\nu}^*E_{\nu,\mu+\nu})P_{\mu+\nu;\ft} \\
		& =E_{\mu,\mu+\nu}^*P_{\mu;\ft}E_{\mu,\mu+\nu}+E_{\nu,\mu+\nu}^*P_{\nu;\ft}E_{\nu,\mu+\nu} \\
		& =E_{\mu,\mu+\nu}^*E_{\mu,\mu+\nu}+E_{\nu,\mu+\nu}^*E_{\nu,\mu+\nu}= I.
	\end{align*}
	Therefore, $(\mu+\nu)_{\ft}=\mu+\nu$ is type $\ft$.

	(c) Define $U_{\mu}:\mathbb{H}^2_d(\mu)\to \mathbb{H}^2_d(\mu_\ft)\oplus \mathbb{H}^2_d(\mu_\fu)$ by setting $U_\mu h=E_{\mu_\ft , \mu}h\oplus (E_{\mu_\fu, \mu}h)$ for $h\in\mathbb{H}^2_d(\mu)$.
	Then it follows from comments following Definition \ref{D:Types} that $U_{\mu}$ is a surjective isometry.
	The surjective isometry $U_{\lambda}:\mathbb{H}^2_d(\lambda)\to \mathbb{H}^2_d(\lambda_\ft)\oplus \mathbb{H}^2_d(\lambda_\fu)$ is defined similarly.
	We note that, with respect to this direct sum decomposition,
	\ba U_\mu E_{\mu, \la} & = & \bbm E_{\mu _\ft , \mu} \\ E_{\mu _\fu , \mu}  \ebm E_{\mu ,\la}  = \bbm E_{\mu _\ft , \la} \\ E_{\mu _\fu , \la} \ebm \nn \\
	& = & \bbm E_{\mu _\ft , \la _\ft} E_{\la _\ft, \la} \\ E_{\mu _\fu , \la _\fu} E_{\la _\fu , \la} \ebm  = \begin{bmatrix} E_{\mu_\ft , \lambda_\ft} & 0 \\ 0 & E_{\mu_\fu , \lambda_\fu}\end{bmatrix} \bbm E_{\la_\ft, \la} \\ E_{\la _\fu , \la} \ebm  \nn \\
	& = & \begin{bmatrix} E_{\mu_\ft , \lambda_\ft} & 0 \\ 0 & E_{\mu_\fu , \lambda_\fu}\end{bmatrix}U_\lambda. \nn \ea
	Thus,
	\[ E_{\mu, \la}U_\lambda^*=U_\mu^*\begin{bmatrix} E_{\mu_\ft , \lambda_\ft} & 0 \\ 0 & E_{\mu_\fu , \lambda_\fu}\end{bmatrix}. \]
	Let $C_\mu:\mathbb{H}^2_d(\mu)\to\mathbb{H}^2_d(\mu_\ft)\oplus \mathbb{H}^2_d(\mu_\fu)$ be defined by $C_\mu h=E_{\mu_\ft , \mu}h\oplus 0$, with $C_\lambda$ similarly defined.
	Then
	\[ P_{\mu;\ft}E_{\mu, \la}=C_{\mu}^*U_{\mu}E_{\mu, \la}=C_\mu^*\begin{bmatrix} E_{\mu_\ft , \lambda_\ft} & 0 \\ 0 & E_{\mu_\fu, \lambda_\fu}\end{bmatrix}U_\lambda = E_{\mu_\ft ,\mu}^*E_{\mu_\ft ,\lambda_\ft}E_{\lambda_\ft, \lambda}  \]
	and
	\[ E_{\mu, \la}P_{\lambda;\ft}=E_{\mu, \la}U_\lambda^*C_\lambda=U_\mu^*\begin{bmatrix} E_{\mu_\ft , \lambda_\ft} & 0 \\ 0 & E_{\mu_\fu, \lambda_\fu}\end{bmatrix}C_\lambda= E_{\mu_\ft ,\mu}^*E_{\mu_\ft ,\lambda_\ft}E_{\lambda_\ft, \lambda} . \]
	Therefore, $P_{\mu;\ft}E_{\mu, \la}=E_{\mu, \la}P_{\lambda;\ft}$.

	(d) Since $\mu$ is type $\ft$, we have $P_{\mu;\ft}=I$ and so $E_{\mu, \la}P_{\la;\ft}=E_{\mu, \la}$.
	Let $c\in\mathscr{A}_d$ be positive semi-definite.
	By Lemma \ref{L:PosSemi}, we have
	\[ E_{\mu, \la}^*\pi_\la(c)E_{\mu, \la}\leq \pi_\la(c), \]
	and so
	\begin{eqnarray*}
		\mu(c)&=& \ip{I + N_\la}{E_{\mu, \la}^*\pi_\mu(c)E_{\mu, \la}(I + N_\la)} \\
		&=& \ip{P_{\la;\ft}(I + N_\la)}{E_{\mu, \la}^*\pi_\mu(c)E_{\mu, \la}P_{\la;\ft}(I + N_\la)} \\
		& \leq & \ip{I + N_\la}{P _{\la ; \ft} \pi_\la(c)(I + N_\la)} \\
		&=& \la_\ft(c).
	\end{eqnarray*}
	That is, $\mu\leq\la_\ft$.
\end{proof}

\begin{lemma}
	Suppose that $\ft, \fu , \fw \in\mathrm{Types}$ and that $E_{\mu, \la}P_{\lambda;\ft}=P_{\mu;\ft}E_{\mu, \la}$ for all $\mu,\lambda\in\posncm$ of type $\fw$ for which $\mu\leq\lambda$. Further assume that $P_{\nu;\ft}^\bot=P_{\nu;\fu}$ for all $\nu\in\posncm$ of type $\fw$.
	Then the following assertions hold:
	\begin{enumerate}
		\item[(i)] If $\nu_1,\nu_2,\mu\in\posncm$, of type $\fw$, are such that $\nu_1+\nu_2=\mu$ and $\nu_1$ and $\nu_2$ are type $\ft$ and $\fu$, respectively, then $\nu_1=\mu_\ft$ and $\nu_2=\mu_\fu$.
		\item[(ii)] For any $\mu,\lambda\in\posncm$ of type $\fw$, one has $(\mu+\lambda)_\ft=\mu_\ft+\lambda_\ft$.
	\end{enumerate}
	\label{L:DirectSum}
\end{lemma}
\begin{proof}
	(i) Plainly $\nu_1\leq\mu$ and $\nu_2\leq \mu$.
	It follows from Proposition \ref{L:EProperties}(d) that $\nu_1\leq\mu_\ft$ and $\nu_2\leq \mu_\fu$ since
	\ba E_{\mu , \la} P_{\la ; \fu} & = & E_{\mu ,\la} - E_{\mu ,\la } P_{\la ; \ft} \nn \\
	& = & E_{\mu , \la } - P_{\mu ; \ft} E_{\mu , \la } \nn \\
	& = & (I - P_{\mu ; \ft} ) E_{\mu , \la} = P_{\mu ; \fu} E_{\mu , \la}. \nn \ea
	For any positive semi-definite $c\in\mathscr{A}_d$, set
	\[ \delta_1=\mu_\ft(c)-\nu_1(c),\quad \delta_2=\mu_\fu(c)-\nu_2(c). \]
	Note that $\delta_1,\delta_2$ are non-negative real numbers.
	As
	\[ 0=\nu_1(c)+\nu_2(c)-\mu(c)=-(\delta_1+\delta_2), \]
	it follows that $\delta_1=\delta_2=0$.
	As every element of $\mathscr{A}_d$ is a linear combination of positive semi-definite elements,
	assertion (i) is proved.

	(ii) It follows from Proposition \ref{L:EProperties}(b) that $\mu_\ft+\lambda_\ft$ is type $\ft$ and $\mu_\fu+\lambda_\fu$ is type $\fu$.
	As $(\mu_\ft+\lambda_\ft)+ (\mu_\fu+\lambda_\fu) =\mu+\lambda$, it follows from (i) that
	\[ \mu_\ft+\lambda_\ft=(\mu+\lambda)_\ft. \] 
\end{proof}
\begin{remark}
	Let $\ft,\fu, \fw \in\mathrm{Types}$ be such that $ P_{\mu ; \fw} =P_{\mu;\ft}+P_{\mu;\fu}$ for all $\mu\in\posncm$.
	It follows from Proposition \ref{L:EProperties} that the following assertions are equivalent.
	\begin{enumerate}
		\item[(i)]  $E_{\mu, \la}P_{\lambda;\ft}=P_{\mu;\ft}E_{\mu, \la}$ whenever $\mu,\lambda\in(\mathscr{A}_d)^{\dagger}_+$ are of type $\fw$ and $\mu\leq\lambda$.
		\item[(ii)] $\mu_\ft\leq\lambda_\ft$ and $\mu_\fu\leq\lambda_\fu$ whenever $\mu,\lambda\in(\mathscr{A}_d)^{\dagger}_+$ are of type $\fw$ and $\mu\leq\lambda$.
	\end{enumerate}
	Indeed, that (ii) implies (i) it precisely Proposition \ref{L:EProperties}(c).
	In the other direction, we first note that $\mu_\ft\leq \lambda$ and $\mu_\fu\leq\lambda$.
	Assume (i).
	Since $\mu_\ft$ and $\mu_\fu$ are type $\ft$ and $\fu$, respectively, it then follows from Proposition \ref{L:EProperties}(d) that $\mu_\ft\leq\lambda_\ft$ and $\mu_\fu\leq\lambda_\fu$. In particular, (i) and (ii) hold in the case where $\fw =\mr{all}$, in which case our starting assumption is that $P_{\mu ; \fw } = I = P_{\mu ; \ft} + P_{\mu ; \fu}$. 
\end{remark}

\begin{prop} \label{embedecomp}
Suppose that $\ga, \la \in \posncm$ and $\ga \leq \la$. Let $E:= E _{\ga , \la} : \mathbb{H} ^2 _d (\la ) \to \mathbb{H} ^2 _d (\ga )$ be the  contractive co-embedding. Then $E P_{\la _{ac} } = P_{\ga _{ac}} E$ and $E P_{\la _s} = P _{\ga _s} E$. That is $\ft=\mr{ac}$ and $\fu =\mr{s}$ are hereditary types and determine positive hereditary cones.
\end{prop}
\begin{proof}
	By \cite[Corollary 8.8]{JM-ncld}, if $\ga = \ga _{ac} + \ga _s$ and $\la = \la _{ac} + \la _s$ are the NC Lebesgue decompositions of $\ga , \la$, then $\ga _s \leq \la _s$.
	Since $\la = \ga + (\la - \ga )\geq \ga$, it follows from \cite[Corollary 8.14]{JM-ncld} that $\la _{ac} = \ga _{ac} + ( \la - \ga ) _{ac}\geq \ga_{ac}$.
	Thus, $\la _{ac} \geq \ga _{ac}$ as well.
	The proposition now follows from Proposition \ref{L:EProperties}(c).
\end{proof}

\begin{cor} \label{ACcommute}
With $\ga \leq \la $ as above, if $D= E^* E$, where $E:= E_{\ga , \la}$, then $D P_{\la ; ac} = P_{\la ; ac} D$.
\end{cor}

In the next lemma, recall that if $Q_\la$ is the structure projection of $\Pi _\la$, that $Q_\la ^\perp = Q _{\la; \mr{WC}}$ is the projection onto $\rm{WC} (\Pi _\la)$ by Theorem \ref{T:DLP}.

\begin{lemma} \label{ACtwine}
	Suppose $\mu,\la\in\posncm$ with $\mu \leq \la$.
	Let $Q_{\la}$ and $Q_{\mu}$ be the structure projections of $\Pi_\lambda$ and $\Pi_\mu$, respectively.
	Then, $E_{\mu, \la }Q _{\la}^\bot = Q_{\mu}^\bot E_{\mu, \la} Q _{\la}^\bot$.
\end{lemma}
\begin{proof}
	Set $E=E_{\mu, \la}$.
	Let $h\in\mathbb{H}^2_d(\lambda)$ be a WC vector of $\Pi_\lambda$.
	By Theorem \ref{T:DLP} there is an intertwiner $X:\mathbb{H}^2_d\to \mathbb{H}^2_d(\lambda)$ and a vector $g\in \mathbb{H}^2_d$ such that $Xg=h$.
	As $EX$ intertwines $\Pi_\lambda$ and $\Pi_\mu$, it follows that $Eh=EXg$ is a WC vector of $\Pi_\mu$.
	Thus, $EQ_{\lambda}^\bot=Q_{\mu}^\bot EQ_{\lambda}^\bot$.
\end{proof}

\begin{lemma} \label{Cuntzhered}
	Suppose $\mu,\la\in\posncm$ with $\mu \leq \la$.
	Let $P _\la := P _{\la ; Cuntz}$ be the $\Pi _\la -$reducing projection onto the support of its Cuntz direct summand.
	Then, $E_{\mu, \la} P_\la = P _\mu E_{\mu, \la} P _\la$.
	In particular, if $\la$ is Cuntz type, then $\mu$ is Cuntz type, and $\ft=\mr{Cuntz}$ is a hereditary type.
\end{lemma}
\begin{proof}
	It follows from Popescu's Wold decomposition theorem, \cite[Theorem 1.3]{Pop-dil}, that the range of $P_\la$ is the set of all $x \in \mathbb{H} ^2 _d (\la )$ so that for any non-negative integer, $N$, there exist $\{x_\alpha | \ \alpha\in\F ^d,|\alpha|=N\} \subset \mathbb{H} ^2 _d (\la )$ such that 
$$ x = \sum _{|\alpha | = N} \Pi _\la ^\alpha x_\alpha. $$
Set $E=E_{\mu, \la}$.
Then,
$$ E x = \sum _{|\alpha | = N} \Pi _\mu ^\alpha E x_{\alpha},$$
for any non-negative integer $N$.
It follows that $E P_\la = P _\mu E P _\la$, from which the remaining claim follows on application of Proposition \ref{L:EProperties}(a).
\end{proof}

\begin{lemma} \label{vNvsdil}
Let $\Pi$ be a row isometry on a Hilbert space $\cH$ and set 
$$ \cH _0 := \bigvee _{\beta,\gamma \in \mathbb{F}^d} \Pi^\beta \Pi^{\gamma *} \mr{WC}(\Pi ). $$
Then $\cH _0$ is $\Pi -$reducing and the restriction of $\Pi$ to $\cH_0^\perp$ is the von Neumann--type summand of $\Pi$. 
\end{lemma}
\begin{proof}
By a result of M. Kennedy, a row isometry, $\Pi$, is of von Neumann type if and only if it has no wandering vectors \cite{MK-wand}. Specifically, any pure or AC Cuntz row isometry has wandering vectors. Since a dilation--type row isometry has a pure type$-L$ restriction to the non-trivial invariant subspace of its weak$-*$ continuous vectors, it also has wandering vectors. It is clear that any wandering vector for $\Pi$ belongs to $WC (\Pi )$. By construction $\cH_0^\perp \cap WC (\Pi ) = \{ 0 \}$ so that $\Pi$ restricted to $\cH _0 ^\perp$ has no wandering vectors and is hence of von Neumann type.
That is, $\mathcal{H}_0^\bot\subset \nbran P_{\mr{vN}}$.
Let $h\in \nbran P_{\mr{vN}} $, let $v\in\mathbb{H}^2_d$, and let $X:\mathbb{H}^2_d\to \mathcal{H}$ be an intertwiner.
Then, for any words $\beta,\gamma$, we have
\[ \ip{\Pi^\beta \Pi^{\gamma *} Xv}{h}=\ip{\Pi^\beta \Pi^{\gamma *} Xv}{P_{\mr{vN}}h}=\ip{\Pi^\beta\Pi^{\gamma *} (P_{\mr{vN}}X)v}{h}. \]
Since $P_{\mr{vN}}X v\in \mr{WC}(\Pi)\cap \nbran P_{\mr{vN}}$, we have $P_{\mr{vN}}Xv=0$ and thus $h\in\mathcal{H}_0^\bot$.
Therefore, $\mathcal{H}_0^\bot= \nbran P_{\mr{vN}}$.
\end{proof}
\begin{remark}
	If $\Pi$ is of dilation type on $\cH$, then $\mr{WC}(\Pi)$ is $\Pi -$invariant but cannot contain any $\Pi -$reducing subspace.
	Thus, $\cH _0 = \cH$ for dilation--type row isometries.
\end{remark}
\begin{remark}
	By Theorem \ref{T:DLP}, we have $\mathrm{WC}(\Pi)$ equal to the range of $Q^\bot$, where $Q$ is the structure projection of $\Pi$.
	Suppose $\Pi$ and $\Xi$ are unitarily equivalent row isometries.
	That is, there exists $U$ a surjective isometry such that $U\Pi^\alpha=\Xi^\alpha U$ for each word $\alpha$.
	It then follows from Lemma \ref{vNvsdil} that $UP_{\Pi;\rm{vN}}=P_{\Xi;\rm{vN}}U$.
\end{remark}

The following fact is well--known and can be found in \cite[Lemma 8.9]{JM-ncld}:
\begin{lemma} \label{Cuntztwine}
	Let $\Pi $ and $\Xi$ be row isometries on Hilbert spaces $\cH , \cJ$ respectively and suppose that $\Pi$ is a Cuntz unitary.
	If $X : \cH \rightarrow \cJ$ is a bounded linear map satisfying
		$$ X\Pi^\alpha=\Xi^\alpha X, \quad \alpha\in\F ^d, $$
	then
		$$ \Pi ^\alpha X^* = X^*\Xi^\alpha, \quad \alpha\in\F ^d,$$
	and $X^* X $ is in the commutant of the von Neumann algebra of $\Pi$, and similarly $XX^*$ is in the commutant of the von Neumann algebra of $\Xi$.
\end{lemma}

\begin{lemma}\label{vNhered}
	Suppose $\mu,\la\in\posncm$ satisfy $\mu \leq \la$.
	If $\la$ is of von Neumann type, then so is $\mu$. That is, $\ft = \mr{vN}$ is a hereditary type.
\end{lemma}
\begin{proof}
	The set of all positive singular NC measures is a positive hereditary cone so that $\mu$ is necessarily singular.
	That is, $\mu = \mu _{\rm{dil}} + \mu _{\rm{vN}}$, is the sum of a positive dilation--type and a von Neumann--type NC measure.
	Suppose that $x \in \mr{WC}(\mu )\subseteq \hardy (\mu _{\rm{dil}})$ is a weak$-*$ continuous vector.
	By Theorem \ref{T:DLP}, there is a bounded intertwiner $X : \mathbb{H} ^2 _d \rightarrow \mathbb{H} ^2 _d (\mu )$ and a vector $f \in \mathbb{H} ^2 _d$ so that $Xf =x$.
	Since $\mu \leq \la$, the  co-embedding $E : \mathbb{H} ^2 _d (\la ) \to \mathbb{H} ^2 _d (\mu )$ is contractive and $E \Pi _\la ^\alpha = \Pi _\mu ^\alpha E$ for any word $\alpha$.
	By Lemma \ref{Cuntztwine}, we also have that $E^* \Pi _\mu ^\alpha = \Pi _\la ^\alpha E^*$, so that $Y := E^* X : \mathbb{H} ^2 _d \rightarrow \mathbb{H} ^2 _d (\la )$ is an intertwiner:
	$$
		YL^\alpha = E^* X L^\alpha = E^* \Pi _\mu ^\alpha X = \Pi _\la ^\alpha Y.	
	$$ 
	Setting $y = Y f \in \mathbb{H} ^2 _d (\la )$, we see that $y$ is in the range of a bounded interwtiner.
	Thus, $y \in \rm{WC} (\la ) = \{ 0 \}$, since $\la$ is of von Neumann type.
	Because $E^*$ is injective, we have $x=0$.
	It follows that $\rm{WC} (\mu ) = \{ 0 \}$, and thus $\mu _{\rm{dil}}= 0$.
	We conclude that $\mu$ is of von Neumann type.
\end{proof}

\begin{lemma} \label{diltwine}
	Suppose $\mu,\lambda\in\posncm$ satisfy $\mu \leq \la$.
	Then, $E_{\mu, \la}P_{\la ; \rm{dil}} = P_{\mu ; \rm{dil}}  E_{\mu, \la}P _{\la ; \rm{dil} }$.
\end{lemma}
The above lemma and Proposition \ref{L:EProperties}(a) imply that $\ft=\mr{dil}$ is also a hereditary type.
\begin{proof}
	Set $P _\la := P _{\la ; \rm{dil}}$, $P_\mu := P _{\mu ; \mr{dil}}$ and $E=E_{\mu, \la}$.
	We know that $P_\la = P _{\la ; \rm{dil}} \leq P_{\la ;\rm{s}}$, that $E P _{\la ;\rm{s}} = P _{\mu ;\rm{s} } E$ and that $E \Pi_{\la}^\beta = \Pi_{\mu}^\beta E$ for every word $\beta$.
	Hence we assume, without loss in generality, that both $\mu $ and $ \la$ are singular.
	The GNS row isometry of any singular NC measure is Cuntz, so we assume in particular that $\Pi _\la , \Pi _\mu$ are Cuntz.
	Note that $\mathrm{WC}(\la ) \subset \nbran P_\la$ is $\Pi_\la -$invariant and that $\nbran P_\la$ is the smallest $\Pi_\la-$reducing subspace of $\mathbb{H}^2_d(\la)$ which contains $\mathrm{WC}(\la)$ by Lemma \ref{vNvsdil}.
	Let $x \in \mathbb{H}^2_d (\la _{\rm{dil}} )$, and denote by $\pi_\la$ and $\pi_\mu$ the GNS representations of $\mathscr{E}_d$ induced by $\la$ and $\mu$, respectively.
	Since $x \in \hardy (\la _{\rm{dil}})$, it belongs to $\nbran P_{\la}$, where $P_\la = P_{\la ; \mr{dil}}$ and thus by Lemma \ref{vNvsdil}, there is a sequence of operators $A_1,A_2,\ldots\in \mathscr{E}_d$ and a $y\in\mathrm{WC}(\la)$ such that
		$$ x = \lim_n \pi_\la(A_n) y. $$
	Since $\Pi_\la$ is Cuntz, we can again apply Lemma \ref{Cuntztwine} to find that  
		$$ E x = \lim_n \pi_\mu(A_n) Ey. $$
	Because $y \in \mathrm{WC}(\la )$, it follows from Lemma \ref{ACtwine} that $Ey \in \mr{WC}(\mu )$.
	Thus, $Ex \in \nbran P_\mu$, where $P_\mu = P _{\mu ; dil}$ by Lemma \ref{vNvsdil} again.
\end{proof}

\begin{lemma} \label{munotvN}
	Let $\lambda\in\posncm$.
	If $x \in \nbran P_{\la;\mr{vN}}$, then the positive NC measure $\lambda_x$ determined by $\la _x (L^\alpha ) = \ip{x}{\Pi _\la ^\alpha x} _\la$, $\alpha\in\F ^d$, is of von Neumann type.
\end{lemma}
\begin{proof}
	Let $\mathcal{H}_x$ denote the cyclic subspace of $\mathbb{H}^2_d(\lambda)$ generated by $x$.
	Note that there is a surjective isometry $U:\mathbb{H}^2_d(\lambda_x)\to \mathcal{H}_x$ such that $U(a+N_{\lambda_x})=\pi_\lambda(a)x$ for each $a\in\A$.
	Because $\lambda$ is von Neumann type, for any given word $\beta$, there exists a net $(a_\gamma)_\gamma$ in $\A$ such that $(\Pi^\beta_\lambda)^*$ is the weak$-*$ limit of $(\pi_\lambda(a_\gamma))_\gamma$.
	As $U^*\pi_\lambda(a)U=\pi_{\lambda_x}(a)$ for any $a\in\A$, it follows that $(\Pi^\beta_{\lambda_x})^*$ is the weak$-*$ limit of $(\pi_{\lambda_x}(a_\gamma))_\gamma$.
	This shows that the weak$-*$ closure of $\pi_{\lambda_x}(\A)$ is self-adjoint, and thus $\Pi_{\lambda_x}$ is of von Neumann type.
\end{proof}

\begin{lemma}\label{L:vNContain}
	Let $\mu,\nu\in\posncm$.
	If $\nu\leq \mu$ and $\nu$ is of von Neumann type,
	then $\nu\leq \mu_{\rm{vN}}$. That is, $\ft = \mr{vN}$ is a hereditary type.
\end{lemma}
\begin{proof}
	By Proposition \ref{embedecomp}, $E_{\nu, \mu}P_{\mu;\rm{ac}}=P_{\nu;\rm{ac}}E_{\nu, \mu}$.
	By Lemma \ref{diltwine}, $E_{\nu, \mu}P_{\mu;\rm{dil}}=P_{\nu;\rm{dil}}E_{\nu, \mu}P_{\mu;\rm{dil}}$.
	As $\nu$ is of von Neumann type, we know that $P_{\nu;\rm{dil}}=0=P_{\nu;\rm{ac}}$, and thus
	\[ E_{\nu, \mu}=E_{\nu, \mu}(P_{\mu;\rm{ac}}+P_{\mu;\rm{dil}}+P_{\mu;\rm{vN}})=E_{\nu, \mu}P_{\mu;\rm{vN}}. \]
	Thus, for any positive semi-definite $c\in\mathscr{A}_d$, we have
	\begin{eqnarray*}
		\nu(c)&=& \ip{I + N_\nu}{\pi_\nu(c)(I + N_\nu)} \\ 
		&=& \ip{I + N_\mu}{E_{\nu, \mu}^*\pi_\nu(c)E_{\nu, \mu}(I + N_\mu)} \\
		&=& \ip{P_{\mu;\rm{vN}}(I + N_\mu)}{E_{\nu, \mu}^*\pi_\nu(c)E_{\nu, \mu}P_{\mu;\rm{vN}}(I + N_\mu)}.
	\end{eqnarray*}
	By Lemma \ref{L:PosSemi}, we have $E_{\nu, \mu}^*\pi_\nu(c)E_{\nu, \mu}\leq \pi_\mu(c)$, and thus
	\[
		\nu(c)\leq \ip{P_{\mu;\rm{vN}}(I + N_\mu)}{\pi_\mu(c)P_{\mu;\rm{vN}}(I + N_\mu)}=\mu_{\rm{vN}}(c)
	\]
	That is, $\nu\leq \mu_{\rm{vN}}$.
\end{proof}

\begin{thm} \label{dilvNcone}
	Suppose $\mu,\lambda\in\posncm$ satisfy $\mu\leq\lambda$.
	Let $E : \hardy (\la ) \to \hardy (\mu )$ denote the  contractive co-embedding.
	Then, $E P _{\la ;vN} = P_{\mu ; vN} E$, and $E P _{\la ; dil} = P _{\mu ; dil} E$ and the
 sets of positive NC measures of dilation and von Neumann type are positive hereditary cones.
\end{thm}

\begin{proof}
	Assume first that $\mu$ and $\lambda$ are singular.
	Set $x=P_{\lambda;\rm{vN}}E^*(I + N_\mu)\in \mathbb{H}^2_d(\lambda)$.
	Plainly, $\lambda_x\leq\mu\leq\lambda$, where, as before, $\la _x (L ^\om ) := \ip{x}{\Pi _\la ^\om x}_\la$.
	By Lemma \ref{munotvN}, we see that $\lambda_x$ is of von Neumann type,
	and thus by Lemma \ref{L:vNContain} we see that $\lambda_x\leq \mu_{\rm{vN}}$. Then, for any $a \in \A$,
	\[ \ip{ (a+N_\mu)}{EP_{\lambda;\rm{vN}}E^*(a+N_\mu)}=\lambda_x(a^*a)\leq \mu_{\rm{vN}}(a^*a)=\ip{a+N_\mu}{P_{\mu;\rm{vN}}(a+N_\mu)}, \]
	 whence
	\[ E P_{\lambda;\rm{vN}}E^*\leq P_{\mu;\rm{vN}}. \]
	Let $Q$ be the projection onto the range of $E P_{\la ; \rm{vN}}$.
	Applying the Douglas factorization lemma then yields $\nbran E P _{\la ; \rm{vN}}  \subseteq \nbran P _{\mu ; \rm{vN}}$.
	In particular, $Q \leq P _{\mu ; vN}$, and it follows that 
	\be \label{vNtwineeq} P _{\mu ; \rm{vN}} E P _{\la ;\rm{vN}} = P_{\mu ; \rm{vN}} Q E P _{\la ; \rm{vN}} = Q E P _{\la; \rm{vN} } = E P _{\la ; \rm{vN}}. \ee
	As $\mu$ and $\lambda$ are assumed singular, we have $P _{\la ; \rm{dil}} = P_{\la ;\rm{vN}} ^\perp$ (and similarly for $\mu$).
	Then, 
	\begin{align*} P_{\mu ; \rm{vN}} E & =  P_{\mu ; \rm{vN}} E ( P _{\la ; \rm{dil} } + P _{\la ; \rm{vN} } )  && \\ 
		& =  \underbrace{P _{\mu ; \rm{vN}} P_{\mu ; \rm{dil}}}_{=0} E P_{\la ; \rm{dil}} + P_{\mu ; \rm{vN}} E P _{\la ; \rm{vN}} && \mbox{by Lemma \ref{diltwine}}  \\ 
		& =  E P_{\la ; \rm{vN}}  && \mbox{by Equation (\ref{vNtwineeq})}.
	\end{align*}
	As $P_{\lambda;\rm{dil}}=I-P_{\lambda;\rm{vN}}$, the theorem is proved in the case where $\lambda$ and $\mu$ are singular.

	In the general case, where $\lambda$ and $\mu$ are not necessarily singular, we note that $\mu\leq\lambda$ implies $\mu_s\leq \lambda_s$, thus $E_{\mu_{\rm{s}},\lambda_{\rm{s}}}P_{\lambda_{\rm{s}};\rm{vN}}=P_{\mu_{\rm{s}};\rm{vN}}E_{\mu_{\rm{s}}, \lambda_{\rm{s}}}$.
	As seen in the proof of Proposition \ref{L:EProperties}, there are unitary intertwiners $U_{\mu}:\mathbb{H}^2_d(\mu)\to\mathbb{H}^2_d(\mu_{\rm{ac}})\oplus \mathbb{H}^2_d(\mu_{\rm{s}})$ and $U_{\lambda}:\mathbb{H}^2_d(\lambda)\to\mathbb{H}^2_d(\lambda_{\rm{ac}})\oplus \mathbb{H}^2_d(\lambda_{\rm{s}})$ such that
	\[ U_{\mu}EU_{\lambda}^*=\begin{bmatrix} E_{\mu_{\rm{ac}}, \lambda_{\rm{ac}}} & 0 \\ 0 & E_{\mu_{\rm{s}}, \lambda_{\rm{s}}} \end{bmatrix}. \]
	Since $\mu _s = \mu _{\mr{vN}} + \mu _{\mr{dil}}$ and $\hardy (\mu _{\mr{s}}) \simeq \hardy ( \mu _{\mr{vN}} ) \oplus \hardy (\mu _{\mr{dil}})$, it follows that $U_{\mu}P_{\mu;\rm{vN}}U_{\mu}^*=0\oplus P_{\mu_{\rm{s}};\rm{vN}}$ and a similar formula holds for $\lambda$.
	Thus,
	\[ U_{\mu}P_{\mu;\rm{vN}}EU_{\lambda}^*=(U_{\mu}P_{\mu;\rm{vN}}U_{\mu}^*)(U_{\mu}EU_{\lambda}^*)=(U_{\mu}EU_{\lambda}^*)(U_{\lambda}P_{\lambda;\rm{vN}}U_{\lambda}^*)=U_{\mu}EP_{\lambda;\rm{vN}}U_{\lambda}^*. \]
	That is, $P_{\mu;\rm{vN}}E=EP_{\lambda;\rm{vN}}$.
	As $P_{\mu;\rm{ac}}E=EP_{\lambda;\rm{ac}}$, we have
	\[ P_{\mu;\rm{dil}}E=(I-P_{\mu;\rm{ac}}-P_{\mu;\rm{vN}})E=E(I-P_{\lambda;\rm{ac}}-P_{\lambda;\rm{vN}})=EP_{\lambda;\rm{dil}}. \]
	It now follows from Proposition \ref{L:EProperties}(b) that the dilation--type and von Neumann--type positive NC measures form hereditary cones.
\end{proof}

The following result refines our NC Lebesgue decomposition, \cite[Section 8]{JM-ncld}, by further decomposing any positive and singular NC measure into positive dilation--type and von Neumann--type NC measures:

\begin{cor}[NC Kennedy--Lebesgue--von Neumann decomposition]\label{refncld}
	Any positive NC measure $\mu \in \posncm$ has a unique \emph{NC Kennedy--Lebesgue--von Neumann decomposition}, 
	$$ \mu = \mu _{ac} + \mu _{dil} + \mu _{vN}, $$ where $\mu _{ac}, \mu_{dil}, \mu _{vN} \in \posncm$ are positive NC measures of absolutely continuous--type, dilation--type and von Neumann--type, respectively. The absolutely continuous, dilation--type and von Neumann--type positive NC measures each form a positive hereditary cone.
	If $\nu_1,\nu_2,\nu_3$ are, respectively, absolutely continuous, dilation--type and von Neumann--type positive NC measures, and $\mu=\nu_1+\nu_2+\nu_3$, then
	\[ \nu_1=\mu_{\rm{ac}},\quad \nu_2=\mu_{\rm{dil}} \quad \mbox{and} \quad \nu_3=\mu_{\rm{vN}}. \]
	Moreover, if $\ft\in\{\rm{ac},\rm{dil},\rm{vN}\}$, then for any $\nu,\lambda\in\posncm$
	\[ (\nu+\lambda)_\ft=\nu_\ft+\lambda_\ft. \]
\end{cor}
\begin{proof}
	It is already known from \cite{JM-ncld} that the absolutely continuous positive NC measures form a hereditary cone and the fact that the dilation-- and von Neumann--type positive NC measures form hereditary cones was proven in Theorem \ref{dilvNcone}.
	The additivity of $(\cdot)_\ft$ follows from Theorem \ref{dilvNcone} and Lemma \ref{L:DirectSum}(ii).
	That $\nu_1=\lambda_{\rm{ac}}$ follows from Proposition \ref{embedecomp}, leaving $\nu_2+\nu_3=\lambda_{\rm{s}}$.
	From this and Lemma \ref{L:DirectSum}(i), we have $\nu_3=(\lambda_{\rm{s}})_{\rm{vN}}=\lambda_{\rm{vN}}$, the second equality following from Lemma \ref{vNvsdil}.
	It now follows that $\nu_2=\lambda_{\rm{dil}}$.
\end{proof}

\begin{eg}
	In contrast to the types just discussed, the set of Cuntz--type positive NC measures is not a cone. 
	For a univariate example, consider normalized Lebesgue measure $\theta$ on the upper and lower half-circles $\partial \D_+$ and $\partial \D_-$, which sum to normalized Lebesgue measure.
	Both of these measures have the property that their `GNS representations' are unitary.
	One can use this to construct a (perhaps somewhat trivial) multi-variable example by setting 
	\[
		\mu_\pm(L^\alpha)=\begin{cases}
			\int_{\partial \D_\pm}\zeta^k d\theta(\zeta) & L^\alpha=L_1^k \\
			0 & \text{otherwise}, \end{cases}
	\]
	for each word $\alpha$.
	To see that $\mu_\pm$ is Cuntz, it suffices to note that $\Pi_{\mu_{\pm}}$ is unitarily equivalent to $(U_{\pm},0,\ldots,0)$, where $U_{\pm}$ is the unitary operator of multiplication by $\zeta$ on $L^2(\partial \D_{\pm})$.
\end{eg}
\begin{eg}
	On the other hand the set of pure positive NC measures is not hereditary.
	The example \cite[Example 2.11]{DLP-ncld}, based on \cite[Example 3.2]{DP-inv}, provides a row isometry, $\Pi = (\Pi _1 , \cdots , \Pi _d ) : \cH \otimes \C ^d \rightarrow \cH$, on a separable Hilbert space $\cH$ so that $\Pi$ is Cuntz (surjective) and there is a bounded intertwiner, $X : \bH ^2 _d \rightarrow \cH$, 
	$$ X L_k = \Pi _k X, \quad k=1,\ldots,d, $$
	so that $X$ has dense range and $x := X1 \in \cH$ is $\Pi -$cyclic.
	Set $T := X^* X$.
	Now choose any real number $c >0$ and note that $T +c I \geq cI >0$.
	This is a strictly positive left Toeplitz operator which is bounded below, and hence it is factorizable by \cite[Theorem 1.5]{Pop-entropy}.
	That is, there exists an $F\in\mathfrak{R}_d^\infty$ such that $T+cI=F^*F$ such that $g=F1$ is cyclic for $L$.
	Define the positive NC measures $\mu _T (L^\om ) := \ip{1}{TL^\om 1}_{\bH ^2}$ and $\mu _{T+cI} (L^\om ) := \ip{1}{(T+cI) 1}_{\bH ^2}$.
	Then clearly $\Pi _{\mu _T} \simeq \Pi $, where $\Pi$ is the AC Cuntz row isometry considered above and $\mu _{T} \leq \mu _{T +cI}$.
	Since $g$ is $L$-cyclic
	$$ \mu _{T +cI} (L^\om ) = \ip{1}{F^*F L^\om 1}_{\bH ^2} = \ip{g}{L^\om g}_{\bH ^2}, \quad \om\in\F ^d. $$
	Clearly $p + N _{\mu _{T+cI}} \mapsto p(L) g$ extends to an onto isometry intertwining the GNS row isometry of $\mu _{T +cI}$ and $L$ so that $\mu _{T + cI }$ is of pure type$-L$ and $\mu _{T +c I} \geq \mu _T$ where $\mu _T$ is AC Cuntz. 
\end{eg}

\section{Complex NC measures}

Our goal for the remainder of the paper is to apply the preceding results to study analytic (and complex) NC measures. As for positive NC measures we define:
\begin{defn}
An NC measure $\mu \in \ncm $ is \emph{absolutely continuous} (AC) if it has a weak$-*$ continuous extension to the free Toeplitz system, $\scr{T} _d = \scr{A} _d ^{-\text{weak-}*}$.
\end{defn}

Let $m\in\posncm$ be given by $m (b ) = \ip{1}{b 1}$, $b\in\mathscr{A}_d$.
We call $m$ \textit{NC Lebesgue measure}, and note that it is absolutely continuous.
It plays the role of normalized Lebesgue measure in this NC measure theory \cite{JM-ncFatou,JM-ncld}.

In \cite[Theorem 2.10]{DP-inv}, Davidson--Pitts show that any bounded linear functional on $\A$ that extends weak$-*$ continuously to $\A ^{-\text{weak-}*} = L^\infty _d$ is a vector functional, for $d\geq 2$.
The next lemma shows that their proof extends to our setting.

\begin{lemma} \label{ACvecfun}
	Any absolutely continuous NC measure $\mu \in \ncm$, for $d \geq 2$, is a vector functional.
	That is, if $\mu\in\ncm$ is absolutely continuous, then there exist $f,g\in\hardy$ such that
	\[ \mu(b)=\ip{f}{bg},\quad b\in\mathscr{A}_d. \]
\end{lemma}
\begin{proof}
	By general considerations, $\mu$ can be extended (with generally $\eps >0$ increase in norm) to a weak$-*$ continuous linear functional, $\hat{\mu}$, acting on $\mc{L} (\mathbb{H} ^2 _d)$.
	Indeed, since $\scr{T} _d = \scr{A} _d ^{-\text{weak-}*}$ is weak$-*$ closed, it can be identified with the annihilator $\scr{S} ^\perp$ where $\scr{S} \subseteq \mr{Tr} (\bH ^2 _d )$ is a norm-closed subspace of the trace-class operators on $\bH ^2 _d$, the pre-dual of $\scr{L} (\bH ^2 _d )$ \cite[Corollary 2.4.11]{analysisnow}.
	Thus, for any $K \in \scr{S}$ and any $a \in \scr{T} _d$, we have $\mr{tr} (K a ) =0$.
	If $q : \mr{Tr} (\bH ^2 _d ) \rightarrow \mr{Tr} (\bH ^2 _d )/\scr{S}$ is the quotient map, then $q^* : \left( \mr{Tr} (\bH ^2 _d )/\scr{S} \right) ^* \to \scr{L} (\bH ^2 _d )$ can be identified with the inclusion map of $\scr{T} _d$ into $\scr{L} (\bH ^2 _d )$ \cite[Proposition 2.4.13]{analysisnow}.
	That is, the pre-dual of $\scr{T} _d$ is isomorphic to $\mr{Tr} (\bH ^2 _d )/\scr{S}$ and linear functionals on $\scr{A} _d$ which extend weak$-*$ continuously to $\scr{T} _d$ can be identified with this quotient space.
	It follows that we can identify $\mu$ with the equivalence class $K + \scr{S}$ for some $K \in \mr{Tr} (\hardy )$.
	Hence, for any $S \in \scr{S}$ and $a \in \scr{T} _d$,
	$$ \mr{tr} ( (K+S) a ) = \mr{tr} (K a) = \mu (a).$$
	Since
	\[ 
	 \| \mu \| = \inf _{S \in \scr{S}} \| K + S \| _{\mr{Tr} (\hardy )},
	\]
	there exists, for any $\eps >0$, an $S' \in \scr{S}$ so that 
	$$ \| K + S' \| _{\mr{Tr} (\hardy)} \leq \| \mu \| + \eps. $$ 
	Set $K ' = K +S'$.
	Then $\hat{\mu}:\mathscr{L}(\hardy)\to \mathbb{C}$ given by
	$$ \hat{\mu} (A) = \mr{tr} (A K'), \quad A\in\mathscr{L}(\hardy), $$
	is a weak$-*$ continuous extension of $\mu$ to $\scr{L} (\hardy )$ with norm $\| \hat{\mu } \| \leq \| \mu \| + \eps$.
	The trace-class operator $K'$ has a singular-value decomposition
	$$ K ' = \sum_{k=1}^\infty  s_k \ip{x_k}{\cdot} y_k; \quad \quad x_k, y_k \in \mathbb{H} ^2 _d, \ s_k \geq 0.$$ 
	Choose any sequence of words $\om _k \in \F ^d$ so that $\nbran R^{\om _k} \perp \nbran R^{\om _j}$ for $k \neq j$.
	For example, one can choose the words $\om _{k+1} = 2^k 1$ for $k \in \N$.
	Then, 
	$$ x:= \sum_{k=1}^\infty s_k ^{1/2} R ^{\om _k} x_k \quad \mbox{and} \quad y:= \sum_{k=1}^\infty s_k ^{1/2} R^{\om _k} y_k $$
	both converge to elements in $\hardy$.
	In what follows, $\delta_{k,j}=0$ when $k\neq j$ and $\delta_{j,j}=1$ for all $j$.
	For any $a_1, a_2 \in \A$, we have
	\ba \ip{x}{a_1^* a_2 y} & = & \sum_{k,j=1}^\infty s_k^{1/2} s_j^{1/2} \ip{R^{\om _k} x_k}{a_1^* a_2 R^{\om _j} y_j}_{\mathbb{H} ^2 _d} \nn \\
	& = & \sum_{k,j=1}^\infty s_k ^{1/2} s_j ^{1/2} \ip{x_k}{a_1^* \underbrace{R ^{\om _k *} R^{\om _j}}_{=\delta _{k,j} I} a_2  y_j}_{\mathbb{H} ^2 _d} \nn \\
	& = & \sum _{k=1}^\infty s_k \ip{x_k}{a_1^* a_2 y_k}_{\mathbb{H} ^2 _d} \nn \\
	& = & \mr{tr} \left( a_1^* a_2 K \right)  \nn \\
	& = & \hat{\mu } (a_1^* a_2) = \mu (a_1^* a_2). \nn \ea
	The lemma now follows from the fact that $\A^*\A$ is norm dense in $\mathscr{A}_d$, 
\end{proof}

Given $\mu\in\ncm$, define $\mu^*\in\ncm$ by
\[ \mu^*(b)=\overline{\mu(b^*)}, \quad b\in\mathscr{A}_d. \]
We also set
\[ \nbre \mu=\frac{\mu + \mu^*}{2}\quad \quad  \mbox{and} \quad \quad \nbim \mu=\frac{\mu - \mu^*}{2i}. \]

\begin{cor} \label{ACWstock}
	Any absolutely continuous NC measure $\mu \in \ncm$ can be decomposed as $\mu = \left( \mu _1 - \mu _2 \right) +i \left( \mu _3 - \mu _4 \right)$ where each $\mu _k \geq 0$ is AC.
	In particular, $ \mu _1 + \mu _2 + \mu _3 + \mu _4$ is AC.  
\end{cor}
\begin{proof}
	Applying Lemma \ref{ACvecfun} to $\mu$, we obtain vectors $x,y\in\hardy$ such that $\mu(b)=\ip{x}{by}$ for $b\in\mathscr{A}_d$.
	Set $\la := \frac{1}{2} \left( m_x + m_y \right)$, where,
	$$ m_h (b) := \ip{h}{b h}, \quad b\in\mathscr{A}_d,\; h\in\hardy. $$ 
	Then we observe that, for any positive semi-definite $c\in\mathscr{A}_d$, 
	$$ 2 \la (c) \pm 2 \nbre \mu (c) = m _{x \pm y} (c ) \geq 0, $$ 
	and similarly, 
	$$ 2 \la (c) \pm 2 \nbim \mu (c) = m_{x \pm i y } (c) \geq 0. $$ 
	Therefore, we can set $\mu _1 = \la + \nbre \mu $, $\mu _2 = \la - \nbre \mu$, etc. 
\end{proof}

\subsection{General Wittstock decomposition}

Any $\mu \in \ncm$ can be written as a linear combination of four positive NC measures, 
$$ \mu  = \left( \mu _1 - \mu _2 \right) + i \left( \mu _3 - \mu _4 \right); \quad \quad \mu _k \in \posncm, $$ where 
$$ \nbre \mu = \frac{\mu + \mu ^*}{2} = \mu _1 - \mu _2, \quad \mbox{and} \quad \nbim \mu = \frac{\mu - \mu ^* }{2i} = \mu _3 - \mu _4. $$
This also works for operator-valued NC measures, \emph{i.e.} operator-valued completely bounded maps on the free disk system, by the Wittstock decomposition theorem \cite{Wittstock} \cite[Theorem 8.5]{Paulsen}.  

\begin{defn}
	If $\vec{\la} := ( \la _1 , \la _2 , \la _3 , \la _4 ) \subset (\scr{A} _d ) ^{\dag ; \, 4} _+$ is such that $\mu = (\la _1 - \la _2 ) +i (\la _3 - \la _4 )$, then we call $\vec{\la}$ a \emph{Wittstock decomposition of $\mu$}.
	The set of all Wittstock decompositions of $\mu$ is denoted by $\scr{W} (\mu )$.
	Given a Wittstock decomposition $\vec{\la} \in \scr{W} (\mu)$, the \emph{total variation of $\mu$ with respect to $\vec{\la}$} is 
	$$ |\vec{\lambda}| := \la _1 + \la _2 + \la _3 + \la _4 \geq 0. $$ 
\end{defn}
\begin{remark}
	The total variation $|\vec{\la}|$ depends on the choice of Wittstock decomposition of $\mu$ and this is not uniquely determined by $\mu$.
	Indeed, if $\vec{\la} = ( \la _k )_{k=1}^4$ is any Wittstock decomposition of $\mu$, then so is $\vec{\mu}$ where $\mu _1 = \la _1 + \ga$, $\mu _2 = \la _2 + \ga$ and $\la _3 = \mu _3$, $\la _4 = \mu _4$, for any $\ga \in \posncm$.
	This is a rather trivial example of non-uniqueness and in this case $|\vec{\mu}| = |\vec{\la}| + 2 \ga \geq |\vec{\la}|$.
	However, even if $\vec{\mu}, \vec{\la} \in \scr{W} (\mu )$ are two different Wittstock decompositions of $\mu \in \ncm$ so that $|\vec{\la}| \leq |\vec{\mu}|$, this need not imply that $\mu _k \geq \la _k$ for each $k$.
	That is, one could have, for example, that $\mu _1 - \mu _2 = \la _1 - \la _2$ without having $\mu _1 = \la _1 + \ga$ and $\mu _2 = \la _2 + \ga$ for some $\ga \in \posncm$.
\end{remark}

\begin{prop}\label{P:ACiffTVAC}
	Let $\mu\in\ncm$.
	Then, $\mu$ is absolutely continuous if and only if there exists a $\vec{\mu}\in\mathscr{W}(\mu)$ such that $|\vec{\mu}|$ is absolutely continuous.
\end{prop}
\begin{proof}
	If $\mu$ is absolutely continuous, then we can apply Corollary \ref{ACWstock}.
	Conversely, suppose that $\vec{\mu}=(\mu_1,\mu_2,\mu_3,\mu_4)\in\mathscr{W}(\mu)$ is such that $|\vec{\mu}|$ is absolutely continuous.
	Recall that the absolutely continuous NC measures form a hereditary cone.
	As $\mu_j\leq|\vec{\mu}|$, it follows that $\mu_j$ is absolutely continuous for each $j$.
	Thus, $\mu$ is then absolutely continuous.
\end{proof}

\begin{defn}
	Let $\ft\in\{\rm{ac},\rm{s},\rm{dil},\rm{vN}\}$.
	A complex NC measure $\mu \in \ncm$ is type $\ft$, if there exists a Wittstock decomposition $\vec{\mu} = (\mu _k )_{k=1} ^4 \in \scr{W} (\mu )$ such that $|\vec{\mu}| = \mu _1 + \mu _2 + \mu _3 + \mu _4$ is type $\ft$.
\end{defn}
\begin{remark}
	Since $|\vec{\mu}| \geq \mu _k $ for each $1\leq k \leq 4$, where $\vec{\mu} = ( \mu _k )_k$ is a Wittstock decomposition of $\mu$, and the sets of AC, singular, dilation--type and von Neumann--type NC measures are positive hereditary cones, it follows that if $|\vec{\mu}|$ is one of these four types, then so is each $\mu _k$.
	It follows that $\vec{\mu}$ cannot, for example, be both absolutely continuous and von Neumann type without also being $(0,0,0,0)$.
\end{remark}

\begin{lemma}\label{Wittcancel}
	Let $\ft,\fu\in\{\rm{ac},\rm{s},\rm{dil},\rm{vN}\}$ and suppose $(\nu_\ft)_\fu=0$ for all $\nu\in\posncm$.
	If $\mu\in\ncm$ is type $\ft$ and $(\lambda_1,\lambda_2,\lambda_3,\lambda_4)\in\mathscr{W}(\mu)$, then
	\[ \lambda_{1;\fu}=\lambda_{2;\fu}, \quad \lambda_{3;\fu}=\lambda_{4;\fu}. \]
\end{lemma}
\begin{proof}
	Since $\mu$ is type $\ft$, there exist $(\mu_1,\mu_2,\mu_3,\mu_4)\in \mathscr{W}(\mu)$, where each $\mu_j$ is type $\ft$. 
	Thus, $(\mu_j)_\ft=\mu_j$ for each $j$.
	By separating the real and imaginary parts of $\mu$, we find that
	\[ \mu_1-\mu_2=\lambda_1-\lambda_2,\quad \mu_3-\mu_4=\lambda_3-\lambda_4, \]
	and thus
	\[ \mu_1+\lambda_2=\lambda_1+\mu_2,\quad \mu_3+\lambda_4=\lambda_3+\mu_4. \]
	Using Corollary \ref{refncld}, we have
	\[ \lambda_{2;\fu}=(\mu_1)_\fu+(\lambda_2)_\fu=(\mu_1+\lambda_2)_\fu=(\lambda_1+\mu_2)_\fu=\lambda_{1;\fu}. \]
	Likewise, we find that $\lambda_{3;\fu}=\lambda_{4;\fu}$.
\end{proof}
By the preceeding lemma, the condition $(\nu_\ft)_\fu=0$ is symmetric in $\ft,\fu$ and it occurs when $(\ft,\fu)\in\{(\rm{ac},\rm{s}),(\rm{ac},\rm{dil}),(\rm{ac},\rm{vN}),(\rm{dil},\rm{vN})\}$, \emph{i.e.} whenever $P_{\nu ; \ft } P_{\nu ; \fu} = 0$ for every $\nu \in \posncm$.

Any $\mu \in \ncm$ with Wittstock decomposition $\mu = (\mu _1 - \mu _2 ) + i (\mu _3 - \mu _4)$ has a corresponding Lebesgue decomposition:
$$ \mu = \mu _{\rm{ac}} + \mu _s, $$ where
$$ \mu _{\rm{ac}} = (\mu _{1;\rm{ac}} - \mu _{2;\rm{ac}} ) + i (\mu _{3;\rm{ac}} - \mu _{4;\rm{ac}}), $$ and similarly for $\mu_{\rm{s}}$. Similarly,
$ \mu_{\rm{s}} = \mu _{\rm{dil}} + \mu _{\rm{vN}}. $ More generally, for $\ft\in\rm{Types}$, we set
$$ \mu _{\ft} := (\mu _{1;\ft} - \mu _{2;\ft} ) + i ( \mu _{3;\ft} - \mu _{4 ;\ft} ). $$

\subsection{GNS formula}\label{S:GNSF}

Let $\mu\in\ncm$ and let $\vec{\mu} = ( \mu _k )_{k=1}^4$ be a Wittstock decomposition of $\mu$.
Since $|\vec{\mu}| \geq \mu _k$, there exists a corresponding  contractive co-embedding $E _k : \mathbb{H} ^2 _d (|\vec{\mu}|) \to  \mathbb{H} ^2 _d (\mu _k ) $ for $k=1,2,3,4$.
Then, given $a_1,a_2\in\A$,
$$ \mu _k (a_1^*a_2) = \ip{I + N_{|\vec{\mu}|}}{ \pi_{|\vec{\mu}|}(a_1)^*D _k \pi_{|\vec{\mu}|}(a_2)(I + N _{|\vec{\mu}|} ) }; \quad \quad D _k := E_k ^* E_k. $$ 
It then follows that 
\be  \mu (a_1^*a_2 ) = \ip{I + N_{|\vec{\mu}|}}{ \pi_{|\vec{\mu}|}(a_1)T _\mu \pi_{|\vec{\mu}|}(a_2) (I + N _{|\vec{\mu}|} ) }, \label{GNS} \ee where
$$ T_\mu := (D_1 - D_2 ) + i (D_3 - D_4 ) $$
is a $|\vec{\mu}| -$Toeplitz operator, \emph{i.e.}
$$ \pi _{|\vec{\mu}| } ( L_k ^* )  T_\mu \pi _{|\vec{\mu}| } (L_j ) = \delta _{k,j} T _\mu. $$

\section{Analytic NC measures}

\begin{defn}
	An \textit{analytic} NC measure, $\mu \in \ncm$, is any complex NC measure that annihilates $\Anot := \A \setminus \{ I \}$.
\end{defn}
The following is an analogue of \cite[Corollary 2, Chapter 4]{Hoff}.
\begin{thm} \label{BabyFnM}
	If $\mu \in \ncm$ is a complex NC measure that annihilates $\Anot$, then each of $\mu _{\rm{ac}}$, $\mu_{\rm{s}}$, $\mu_{\rm{dil}}$ and $\mu_{\rm{vN}}$ annihilate $\Anot$. 
\end{thm}
\begin{proof}
	Let $\vec{\mu}=(\mu_1,\mu_2,\mu_3,\mu_4)\in\mathscr{W}(\mu)$.
	As in Section \ref{S:GNSF}, if $\la := |\vec{\mu}|$, then
	$$ \mu (L^\alpha ) = \ip{ I + N_\la}{T_\mu \Pi_\la ^\alpha (I + N_\la)}_\la, \quad \alpha\in\F ^d,$$
	where $$ T_\mu = (D_1 - D_2 ) +i (D_3 - D_4); \quad \quad D_k = E_k ^* E_k $$ and each
	$E_k : \mathbb{H} ^2 _d (\lambda ) \to \mathbb{H} ^2 _d (\mu _k)$ is the  contractive co-embedding arising from $\mu_k\leq|\vec{\mu}|$.
	By Proposition \ref{embedecomp}, Theorem \ref{dilvNcone} and Corollary \ref{ACcommute}, it further follows that if $\ft \in \{ \rm{ac}, \rm{s}, \rm{dil}, \rm{vN} \}$, then 
	$$ \mu _{\ft} (L^\alpha ) = \ip{ I + N_\la}{T_{\mu ; \ft } \Pi _\la ^\alpha (I + N_\la)}, \quad \alpha\in\F ^d, $$
	where $T_{\mu ; \ft} = T_\mu P_{\la ; \ft} = P_{\la ;\ft} T_\mu$.
	In particular, since $\mu | _{\Anot} \equiv 0$, we can find a sequence of NC polynomials $(a_n)_n$ in $\A$ such that $a_n  + N_\la \rightarrow P_{\lambda;\ft}(I + N_\lambda)$.
	Then, for any $\alpha \neq \varnothing$,
	\ba 0 & = & \lim_{n\to\infty} \mu \left( L^\alpha a_n \right) \nn \\
	& = & \ip{I + N_\la}{T_\mu \Pi _\la ^\alpha P_{\la ;\ft}( I + N_{\la} ) } _\la \nn \\
	& = & \ip{I + N_\la}{T_\mu P_{\la;\ft} \Pi _\la ^\alpha ( I + N_{\la}) } _\la \nn \\
	& =& \mu _{\ft} (L^\alpha ). \nn \ea
	This proves that $\mu _{\ft}$ also annihilates $\Anot$ for any $\ft \in \{ \rm{ac} , \rm{s}, \rm{dil}, \rm{vN} \}$.
\end{proof}

In the next lemma, we say that $\mu \in \ncm$ is of \textit{Cuntz--type} if it has a Wittstock decomposition $( \mu _k ) _{k=1} ^4$ whose corresponding total variation $ \mu _1 + \mu _2 + \mu _3 + \mu _4$ is Cuntz--type.
It follows from Lemma \ref{Cuntzhered} that each $\mu _k$ is of Cuntz--type.

\begin{lemma} \label{bwshift}
	If $\mu \in \ncm$ is a complex NC measure of Cuntz--type that annihilates $\Anot$, then it also annihilates $\A$.
\end{lemma}
\begin{proof}
	Let $\vec{\mu}\in\mathscr{W}(\mu)$ and set $\la := |\vec{\mu}|$.
	By hypothesis, $\lambda$ is Cuntz.
	Since $\Pi _\la$ is Cuntz, $I + N _\la$ is the limit of a sequence of equivalence classes of NC polynomials $a_n \in \Anot$, as follows from \cite[Theorem 6.4]{JM-freeCE}.
	With $T_\mu$ as in Section \ref{S:GNSF}, we have
	$$ \mu  (I) = \lim_{n\to\infty} \ip{ I + N _\la }{ T_\mu \, (a_n  + N _\la) } = \lim_{n\to\infty} \mu  (a_n) =0. $$
\end{proof}

\begin{remark} \label{classproof}
	At this point, the proof of the classical F\&M Riesz theorem, as presented in \cite[Chapter 4]{Hoff}, is straightforward.
	Given any complex measure $\mu $ obeying the above assumptions we have that 
	$$
		\int _{\partial \D } a(\zeta ) \mu _s (d\zeta ) =0,
	$$ for any $a \in \mc{A} (\D )$, $\mc{A} (\D) = \mathbb{A} _1$.
	We then consider the complex measure $\mu _s ^{(\ov{\zeta})} (d\zeta ) := \ov{\zeta } \mu _s (d\zeta ) $.
	This is again a complex singular measure which annihilates $\mc{A} (\D ) ^{(0)} = \mc{A} (\D ) \sm I$.
	By the above lemma $\mu$ also annihilates $\mc{A} (\D)$. In particular it annihilates $1$, so that by construction
	$$
		\int _{\partial \D} \zeta ^k \mu _s (d\zeta ) =0; \quad \quad k \in \{  -1, 0 , 1 ,2 , ... \}.
	$$
	Proceeding inductively, we conclude that all moments of $\mu _s$ vanish so that $\mu _s \equiv 0$.
	In the NC setting, this argument breaks down for singular NC measures of dilation--type; see Section \ref{diltypesect}.
\end{remark}

The main result of this section is an analogue of the F\&M Riesz theorem:

\begin{thm}{ (NC F\&M Riesz Theorem)} \label{ncFnM}
	Every analytic NC measure $\mu \in \ncm$ has vanishing von Neumann part.
\end{thm}

When $d=1$, the von Neumann part of an isometry is the singular part of its unitary direct summand.
Hence any analytic measure on the circle is absolutely continuous and we recover the classical F\&M Riesz theorem with a new proof.
Note, however, that as soon as $d \geq 2$, an analytic NC measure, $\mu$, will be AC if and only if it has no dilation part.
That is, an analytic linear functional on $\scr{A} _d$ need not extend weak$-*$ continuously to $\scr{T} _d$; see Proposition \ref{countereg}.

Note that $X\mapsto L_j^*X$ is a contractive linear map on $\mathscr{L}(\hardy)$.
As $L_j^*(\A+\A^*)\subset \A+\A^*$, it follows by continuity that $b\mapsto L_j^*b$ is a contractive linear endomorphism of $\mathscr{A}_d$.
\begin{defn} \label{bwshiftmeas}
	Let $\lambda\in\ncm$ and $k=1,2,\ldots,d$.
	Define $\lambda^{(k)}\in\mathscr{A}_d$ by
	\[ \lambda^{(k)}(b)=\lambda(L_k^*b), \quad b\in\mathscr{A}_d. \]
\end{defn}
\begin{lemma} \label{posvN}
	Let $a_0\in\A$ and $\lambda\in\posncm$.
	Define $\mu\in\posncm$ by setting $\mu(b)=\lambda(a_0^*ba_0)$ for $b\in\mathscr{A}_d$.
	If $\la$ is absolutely continuous, then so is $\mu$.
	If $\lambda$ is von Neumann type, then so is $\mu$.
\end{lemma}
\begin{proof}
	Plainly, $X\mapsto a_0^*Xa_0$ is a continuous linear map on $\mathscr{L}(\mathbb{H}^2_d)$.
	It is readily verified that $a_0(\mathscr{A}_d)a_0\subset \mathscr{A}_d$, and thus $\mu$ is indeed defined.
	It is an elementary exercise to produce an isometry $V:\mathbb{H}^2_d(\mu)\to\mathbb{H}^2_d(\lambda)$ satisfying
	\[ V(a+N_\mu)=aa_0+N_\lambda, \quad a\in\A. \]
	We see immediately that $\nbran V$ is a closed $\Pi_\lambda-$invariant subspace of $\mathbb{H}^2_d(\lambda)$ and that $\Pi_\lambda|_{\nbran V}$ is unitarily equivalent to $\Pi_\mu$.

	Suppose that $\lambda$ is absolutely continuous.
	Then every element of $\mathbb{H}^2_d(\lambda)$ is a weak$-*$ continuous vector.
	In particular, $V(I + N_\mu)$ is weak$-*$ continuous.
	For any $b\in\mathscr{A}_d$, we note that $V^*\pi_\lambda(b)V=\pi_\mu(b)$, and thus
	\[ \mu(b)=\ip{I + N_\mu}{\pi_\mu(b)(I + N_\mu)}=\ip{V(I + N_\mu)}{\pi_\lambda(b)V(I + N_\mu)}. \]
	Therefore, $\mu$ is absolutely continuous.

	Now suppose instead that $\lambda$ is of von Neumann type.
	Let $\mathfrak{W}_\mu$ and $\mathfrak{W}_\lambda$ denote the weak$-*$ closures of $\pi_\mu(\A)$ and $\pi_\lambda(\A)$, respectively.
	Clearly, $\nbran V$ is $\mathfrak{W}_\lambda$ invariant.
	As $\Pi_\lambda$ is of von Neumann type, it follows that $\mathfrak{W}_\lambda$ is a von Neumann algebra, and thus $\nbran V$ is $\Pi_\lambda-$reducing.
	In particular, $\mathfrak{W}_\lambda|_{\nbran V}$ is a von Neumann algebra.
	As $\mathfrak{W}_\lambda|_{\nbran V }$ is unitarily equivalent to $\mathfrak{W}_\mu$, we see that $\Pi_\mu$ is of von Neumann type.
	That is, $\mu$ is of von Neumann type.
\end{proof}
\begin{remark}
If $\la \in \posncm$ is of dilation--type, it is generally not true that $\la ^{(k)}$ is of dilation--type and hence the classical proof as described in Remark \ref{classproof} breaks down. The next section, Section \ref{diltypesect} provides an example of a dilation--type NC measure, $\xi \in \posncm$, so that $\Xi ^{[2]} := \xi \circ \mr{Ad} _{L_2 ^* , L_2} \in \posncm$ is weak$-*$ continuous and such that $\xi ^{(2)} = \xi (L_2 ^* (\cdot ) ) \in  \ncm$ is analytic but not weak$-*$ continuous, see Proposition \ref{countereg}.
\end{remark}

\begin{prop} \label{keyprop}
	Let $\mu\in\ncm$ and $k\in\{1,2,\ldots,d\}$.
	If $\mu$ is absolutely continuous, then $\mu^{(k)}$ is absolutely continuous.
	If $\mu$ is of von Neumann type, then $\mu^{(k)}$ is of von Neumann type.
\end{prop}
\begin{proof}
	Let $\ft\in\{\rm{ac},\rm{vN}\}$, and let $\mu$ be of type $\ft$.
	First assume that $\mu$ is positive.
	For each $b\in \mathscr{A}_d$, set
	\[ \phi_1(b)=\frac{1}{2}\mu( (I+L_k)^*b(I+L_k)), \quad \phi_2(b)=\frac{1}{2}\mu( (I-L_k)^*b(I-L_k)), \]\[ \phi_3(b)=\frac{1}{2}\mu( (I+iL_k)^*b(I+iL_k)), \quad \phi_4(b)=\frac{1}{2}\mu( (I-iL_k)^*b(I-iL_k)). \]
	Then,
	\ba \phi_1(b)-\phi_2(b)+i(\phi_3(b)-\phi_4(b)) & = & \frac{1}{2} \mu(L_k^*b+bL_k)+\frac{i}{2} \mu(-iL_k^*b+ib L_k) \nn \\
	& = &\mu(b), \nn \ea
	for all $b\in\mathscr{A}_d$ and so $(\phi_1,\phi_2,\phi_3,\phi_4)\in \mathscr{W}(\mu^{(k)})$.
	Because $\mu$ is type $\ft$, it follows from Lemma \ref{posvN} that each $\phi_j$ is type $\ft$, and thus $\mu^{(k)}$ is type $\ft$.

	Now consider the general case of $\mu\in\ncm$.
	Since $\mu$ is type $\ft$, there exists a $(\mu_1,\mu_2,\mu_3,\mu_4)\in\mathscr{W}(\mu)$ such that each $\mu_j$ is of type $\ft$.
	It follows that $\mu_j^{(k)}$ is type $\ft$ for each $j,k$.
	Let $(\phi_{j,\ell})_{\ell=1}^4$ be a Wittstock decomposition of $\mu_j$ where each $\phi_{j,\ell}$ is type $\ft$.
	Then
	\begin{eqnarray*}
		\mu^{(k)} &=& \mu_1^{(k)}-\mu_2^{(k)}+i( \mu_3^{(k)}-\mu_4^{(k)} ) \\
		&=& \phi_{1,1}-\phi_{1,2}+ i(\phi_{1,3}-\phi_{1,4})  - (\phi_{2,1}-\phi_{2,2}+ i(\phi_{2,3}-\phi_{2,4})) \\
		& & +i(\phi_{3,1}-\phi_{3,2}+ i(\phi_{3,3}-\phi_{3,4})) - i(\phi_{4,1}-\phi_{4,2}+ i(\phi_{4,3}-\phi_{4,4})) \\
		& = & (\phi_{1,1}+\phi_{2,2}+\phi_{3,4}+\phi_{4,3}) - (\phi_{1,2}+\phi_{2,1}+\phi_{3,3}+\phi_{4,4}) \\
		& & +i(\phi_{1,3}+\phi_{2,4}+\phi_{3,1}+\phi_{4,2}) - i(\phi_{4,1}+\phi_{3,2}+\phi_{2,3}+\phi_{1,4}). 
	\end{eqnarray*}
	As each $\phi_{j,\ell}$ is of type $\ft$, it follows that $\mu^{(k)}$ has a Wittstock decomposition $\vec{\psi}$ such that $|\vec{\psi}|$ is of type $\ft$, and therefore $\mu^{(k)}$ is of type $\ft$.
\end{proof}
\begin{proof}{ (of Theorem \ref{ncFnM})}
	By Corollary \ref{BabyFnM} and Lemma \ref{bwshift}, if $\mu \in \ncm$ annihilates $\Anot$, then $\mu _{\rm{vN}}$ annihilates $\A$.
	By Proposition \ref{keyprop}, for any $k\in\{1,2,\ldots,d\}$, we see that $\mu _{vN}^{(k)}$ is of von Neumann type.
	Since, by definition, $\mu _{\rm{vN}}^{(k)} = \mu _{\rm{vN}} (L_k^*\cdot)$  and $\mu _{\rm{vN}}$ annihilates $\A$, we have that $\mu^{(k)} _{\rm{vN}}$ is a von Neumann type NC measure which annihilates $\Anot$.
	By Lemma \ref{bwshift} again, $\mu ^{(k)} _{\rm{vN}}$ annihilates $\A$ so that 
	$$ 0 = \mu _{\rm{vN}} ^{(k)} (I) = \mu _{\rm{vN}} (L_k ^* ). $$
	Proceeding inductively we obtain that
	$$ \mu _{\rm{vN}} \left( L^{\alpha *} \right) =0, $$
	for any $\alpha \in \F ^d$ and we conclude that $\mu _{\rm{vN}} \equiv 0$.
\end{proof}

\begin{remark} \label{relate}
	An NC F\&M Riesz theorem was previously obtained in \cite[Theorem A]{CMT-analytic}, by R. Clou\^atre and the second two named authours of the present paper,  using different techniques.
	Both Theorem \ref{ncFnM} and \cite[Theorem A]{CMT-analytic} disallow the presence of von Neumann type summands in their corresponding models of analytic linear functionals on $\mathscr{A}_d$, though what this means in these two cases is different.
	Both papers find that analytic linear functionals on $\mathscr{A}_d$ need not have weak$-*$ continuous extensions to the weak$-*$ closure, $\mathscr{T}_d$, of $\mathscr{A}_d$, as we detail in the next section.
	We also remark that the results of \cite{CMT-analytic} and the current paper describe the obstruction to weak$-*$ continuous extension in different ways. 

	To compare the two sets of results, fix $\lambda\in\ncm$ and suppose $\lambda(\A)=\{0\}$, as this is the definition of analyticity used in \cite{CMT-analytic}.
	There, the free disk system is viewed as embedded, completely isometrically, inside the Cuntz algebra, $\scr{O} _d$, via the quotient map $q:\scr{E}_d\to\scr{O}_d$ whose kernel is the compact operators.
	Let $\Lambda$ be an extension of $\lambda$ to $\scr{O}_d$.
	In \cite[Theorem A]{CMT-analytic}, it is proved that there exists a $*-$representation $\pi:\scr{O}_d\to\scr{L}(\mathcal{H})$, a $\pi(\scr{O}_d)-$cyclic vector $h_1\in\mathcal{H}$ and a vector $h_2\in\mathcal{H}$ such that
	\[ \Lambda(x)=\ip{h_2}{\pi(x)h_1}, \quad x\in\scr{O}_d, \]
	$\|\Lambda\|=\|h_1\|^2=\|h_2\|^2$, and the restriction of $\pi(L):=(\pi(L_1),\ldots,\pi(L_d))$ to the norm closure of $\pi(\A)h_1$ is unitarily equivalent to $L$.
	The triple $(\pi,h_1,h_2)$ is referred to as a ``Riesz representation'' of the functional $\Lambda$.
	The particular form of the representation implies that $\pi(L)$ has no von Neumann--type summand, as is noted in \cite{CMT-analytic}.
	We note that
	\[ \lambda(b)=\ip{h_2}{\pi(q(b))h_1}, \quad b\in\scr{A}_d. \]
	The representation $\pi$ and the vector $h_1$ are such that $|\hat{\Lambda}|(f)=\ip{h_1}{\hat{\pi}(f)h_1}$, where $\hat{\Lambda}$ and $\hat{\pi}$ denote their weak$-*$ continuous extensions to the second dual of $\scr{O}_d$ and $|\hat{\Lambda}|$ is the `radial' part of the polar decomposition a normal linear functional.
	In the present paper, we find that there exists a Wittstock decomposition $\vec{\lambda}$ of $\lambda$ and a $\Pi_{|\vec{\lambda}|}$-Toeplitz operator $T$, such that
	\[ \lambda(a_2^*a_1)=\ip{I + N_{|\vec{\lambda}|}}{\pi_{|\vec{\lambda}|}(a_2)^*T\pi_{|\vec{\lambda}|}(a_1)(I + N_{|\vec{\lambda} |})}, \quad a_1,a_2\in\A, \]
	with $\Pi_{|\vec{\lambda}|}=(\pi_{|\vec{\lambda}|}(L_1),\ldots,\pi_{|\vec{\lambda}|}(L_d))$ having no von Neumann--type summand.
	One will recall that we assume in Theorem \ref{ncFnM} that $\la \in \ncm$ annihilates $\Anot$, not $\A$.
	However, this is equivalent to the analyticity assumptions of \cite[Theorem A]{CMT-analytic}, since a vector functional applied to a $*-$representation of $\scr{O}_d$ annihilates $\Anot$ if and only if it annihilates $\A$; see Lemma \ref{bwshift}.
	
	As another point of contrast,	our total variation $|\vec{\lambda}| = \la _1 + \la _2 + \la _3 + \la _4$ is not uniquely determined by $\lambda$.
	From the polar decomposition $\hat{\Lambda}(x)=|\hat{\Lambda}|(v^*x)$, we may readily produce a Wittstock decomposition $\vec{\nu}$ of $\lambda$, but it is not generally of the type that have proven useful in this paper, nor of course is $|\vec{\nu}|$ equal to $|\hat{\Lambda}|$.
	We also remark that, unless the GNS row isometry of $|\vec{\lambda}|$ is Cuntz, $|\vec{\lambda}|$ need not have a unique positive Kre\u{\i}n--Arveson extension to $\scr{E} _d$ (see the proposition below).
	
	
	For yet another point of contrast, it follows from \cite[Corollary 5.2]{CMT-analytic} that $\lambda$ can be analytic and admit a weak$-*$ continuous extension to $\mathscr{A}_d$, with $\pi(L)$ having a dilation--type summand.
	In the representation of this paper, $\lambda$ admits a weak$-*$ continuous extension to $\mathscr{T}_d$ precisely when $\Pi_{|\vec{\lambda}|}$ is absolutely continuous (for some $\vec{\lambda}\in\mathscr{W}(\lambda)$).
	Despite these differences, we can collect some necessary and sufficient conditions for the absolute continuity of an analytic $\lambda$.
	By combining Proposition \ref{P:ACiffTVAC}, Lemma \ref{ACvecfun} and \cite[Theorem B]{CMT-analytic}, we see that the following conditions are equivalent, where we suppose $\lambda\in\ncm$ is analytic, $\lambda(\Anot)=\{0\}$.
	\begin{enumerate}
		\item[(i)] $\lambda$ extends weak$-*$ continuously to $\scr{T}_d$.
		\item[(ii)] There exists a $\vec{\lambda}\in\mathscr{W}(\lambda)$ such that $|\vec{\lambda}|$ is absolutely continuous.
		\item[(iii)] There exist $f,g\in\mathbb{H}^2_d$ such that $\lambda(b)=\ip{f}{bg}$, $b\in\mathscr{A}_d$.
		\item[(iv)] The weak$-*$ continuous extension of $\lambda$ to the second dual of $\mathscr{A}_d$ annihilates $\{\hat{a}^*\mathfrak{q}-\mathfrak{q}\hat{a}^* | \ a\in\A\}$, where $\mathfrak{q}$ is the free semi-group structure projection of the second dual of $\A$.
	\end{enumerate}
	It should be noted that the equivalences (i)$\Leftrightarrow$(ii)$\Leftrightarrow$(iii) do not require analyticity of $\lambda$.
	For (iv), we remark that $\lambda$ extends weak$-*$ continuously $\scr{T}_d$ if and only if $b\mapsto \lambda(b)-\lambda(I)m(b)$ does, and so the different conditions for analyticity from these two papers do not present any difficulties here.
	
\end{remark}

\begin{prop}
	If $\mu \in \posncm$ is Cuntz type, then $\mu$ has a unique positive Kre\u{\i}n--Arveson extension $\La$ to the Cuntz--Toeplitz algebra, $\scr{E} _d$.
\end{prop}
\begin{proof}
	Let $\La : \scr{E} _d \rightarrow \C$ be any Kre\u{\i}n--Arveson (positive) extension of $\la \in \posncm$.
	Apply the GNS construction to $\left( \La , \scr{E} _d \right)$ to obtain a GNS Hilbert space $L^2 (\La )$ and a $*-$representation $\pi_\La$ satisfying
	$$ \La (x) = \ip{ I + N_\La}{\pi _\La (x) (I + N _\La) }; \quad \quad x \in \scr{E} _d. $$
	By construction, $I + N_\La$ is cyclic for the GNS row isometry $\Pi _\La$.
	Observe that there is an isometry $V:\bH ^2 _d (\la ) \to L^2 (\La )$ with $\Pi_\Lambda$-invariant range.
	Indeed, for any $a\in\A$,
	\[ \| a + N _\la \| ^2  =  \la \left( a^*a \right)  = \La \left( a^*a \right)  =  \| a + N _\La \|^2 . \]
	Thus, there is an isometry $V:\bH^2_d(\la)\to L^2(\La)$ determined by $V(a+N_\lambda)=a+N_\Lambda$, $a\in\A$.
	Plainly, $V\Pi_{\lambda;k}=\Pi_{\Lambda;k}V$ for each $k$.
	
	Next, observe that $V\bH ^2 _d (\la)$ is $\Pi _\La-$reducing and $\Pi _\La | _{V\bH ^2 _d (\la )}$ is unitarily equivalent to $\Pi _\la$.
	Indeed, $\Pi _\la$ is Cuntz, and so any given element $x \in \bH^2 _d (\la )$ is the norm-limit of vectors $x_n + N_\la$, where $x_n\in\Anot$, as shown in \cite[Theorem 6.4]{JM-freeCE}.
	Hence, for any $z \in \scr{E} _d$ and $k=1,2,\ldots,d$, 
	\ba
		\ip{z + N_\La}{\Pi _{\La ;k}^* Vx} & = & \lim_{n\to\infty} \ip{z + N _\la}{\Pi_{\La ;k}^* V(x_n  + N_\la) } \nn \\
		& = & \lim_{n\to\infty} \ip{L_kz + N_\La}{x_n  + N_\La} \nn \\
		& = & \lim_{n\to\infty} \ip{z + N_\La}{(L_k ^* x_n)  + N _\La } \nn \\
		& = & \lim_{n\to\infty} \ip{z + N_\Lambda}{V\Pi_{\lambda;k}^*(x_n+N_\lambda)} \nn \\
		& = & \ip{z + N_\La}{V\Pi_{\la ;k}^*x} . \nn
	\ea
	That is, $\Pi_{\Lambda;k}^*V=V\Pi_{\lambda;k}^*$ for each $k$, whence $V\pi_{\lambda}(y)=\pi_{\Lambda}(y)V$ for each $y\in\scr{E}_d$.

	Finally, note that $I + N_\La = V(I + N _\la)$.
	Since $\pi_\Lambda(\mathscr{E}_d)(I + N_\Lambda)$ is dense in $L^2(\Lambda)$ and $\pi_\lambda(\mathscr{E}_d)(I + N_\lambda)$ is dense in $\bH^2_d(\lambda)$, it follows that $V\bH^2_d(\lambda)=L^2(\Lambda)$, showing that $V$ is in fact a surjective isometry.
	Therefore, for any $z\in\scr{E}_d$,
	\[ \Lambda(z)=\ip{I + N_\lambda}{\pi_\lambda(z)(I + N_\lambda)}. \qedhere \]
\end{proof}

\section{A Dilation-type example}
\label{diltypesect}

Recall that there is, in essence, a bijection between positive, finite and regular Borel measures on the circle and the set of Herglotz functions in the disk, \emph{i.e.} analytic functions in the complex unit disk with positive semi-definite real part.
This correspondence extends to positive NC measures and non-commutative (left) Herglotz functions in $\B ^d _\N$, $\mu \leftrightarrow H_\mu$; see \cite{JM-freeAC,JM-freeCE}.
A fractional linear transformation, the so-called \emph{Cayley transform}, then implements a bijection between the left NC Schur class of contractive NC functions in $\B ^d _\N$ and and the left NC Herglotz class.
If $\mu \in \posncm$ is the (essentially) unique NC measure corresponding to the contractive NC function $b \in [ \mult ] _1$, we write $\mu = \mu _b$, and $\mu _b$ is called the NC Clark measure of $b$; see \cite[Section 3]{JM-ncFatou} for details. 

By \cite[Corollary 7.25]{JM-ncFatou}, if $b \in [\mult ] _1$ is \emph{inner}, \emph{i.e.} an isometric left multiplier, then its NC Clark measure is singular, so that its GNS representation $\Pi _b := \Pi _{\mu _b}$ is a Cuntz row isometry which can be decomposed as the direct sum of a dilation--type row isometry and a von Neumann type row isometry \cite{JM-ncld}. 

Classically, any sum of Dirac point masses is singular with respect to Lebesgue measure on the circle.
Motivated by this,  consider the positive linear functional $\xi \in ( \scr{A} _2^\dag )_+$ defined by 
$$ \xi (L ^\alpha ) = \left\{ \begin{array}{cc} 0 & 2 \in \alpha \\ 1 & 2 \notin \alpha \end{array} \right. ; \quad \quad \alpha \in \F ^2,  $$
and $\xi (I) = 1$.
Here, $2\notin\alpha$ is used to indicate that $\alpha$ does not contain the `letter' $2$.
One may think of $\xi$ as a `Dirac point mass' at the point $(1,0) \in \partial\B ^2_1$, where $\B^2_1$ is the first level of the NC unit ball $\B^2_\N$.
Setting $Z:=(1,0)$, we note that $\xi(L^\alpha)=Z^\alpha$ for all words $\alpha$.
Since $Z$ is a row contraction, it follows from results of Popescu that the map $\xi$ extends to a positive linear functional on $\scr{A} _2$ \cite[Theorem 2.1]{Pop-NCdisk}. 

Before continuing, we remark that the example of this section is related to \cite[Example 2]{CMT-analytic}, which is itself related to atomic representations of \cite{DP-inv}. This example is also a special case of \cite[Example 5.1]{DKS-finrow}.
However, we here choose to start from the linear functional, $\xi$, rather than the functional's representation as a vector functional on a representation, to emphasize the NC function theory associated with the NC measure.

\begin{claim}
	$L_2 +N_\xi$ is a wandering vector for $\Pi_\xi$ and $\Pi_\xi$ has vanishing von Neumann part.
\end{claim}
\begin{proof}
	Note that any wandering vector for $\Pi _\xi$ is always a weak$-*$ continuous vector.
	Indeed, if $w$ is wandering for $\Pi _\xi$, then 
	\ba
		\xi _w (L^\alpha ) & = & \ip{w}{\Pi_\xi ^\alpha w}_{\mathbb{H} ^2 _2 (\xi)}  \nn \\
		& = & \| w\| ^2 _{\mathbb{H} ^2 _2 (\xi)} \delta _{\alpha, \varnothing} = \| w \| ^2 m (L^\alpha), \nn
	\ea
	is a constant multiple of NC Lebesgue measure, and hence is absolutely continuous. 
	
	To see that $L_2 + N_\xi$ is wandering for $\Pi _\xi$, note that
	$$ \ip{L_2 +N_\xi }{\Pi _\xi ^\alpha (L_2 + N_\xi )} = \xi ( L_2 ^* L^\alpha L_2 ) = \delta _{\alpha, \varnothing}. $$
	However $L_2 + N_\xi$ is also $*-$cyclic for $\Pi _\xi$ since $\pi _\xi (L_2 ) ^* (L_2 +N_\xi) =I + N_\xi$ , which is cyclic for $\Pi _\xi$.
	This means that the smallest reducing subspace that contains the weak$-*$ continuous vector $L_2 + N_\xi$ is all of $\mathbb{H} ^2 _2 (\xi)$, so that $\mathbb{H} ^2 _2 (\xi _{\rm{vN}} ) = \{ 0 \}$ by Lemma \ref{vNvsdil}. 
\end{proof}

In what follows, if $\om = i_1 \cdots i_n \in \F ^d$, $i_k \in \{ 1 , \cdots , d \}$, is any word, then we set $\om ^\mrt := i_n \cdots i_1$.
This letter reversal map is an involution on the free monoid.

\begin{claim}
	The NC measure $\xi$ is the NC Clark measure of $b_\xi (Z) = Z_1$, $Z\in\mathbb{B}^2_{\mathbb{N}}$; a left-inner NC function. 
	Thus, $\Pi _\xi$ is purely of dilation--type.
\end{claim}
\begin{proof}
	Given $Z \in \B ^2_n$, let $Z \otimes L ^* := Z_1 \otimes L_1 ^* + \cdots + Z_d \otimes L_d ^*$.
	The (left) Herglotz function, $H_\xi$ of $\xi$ is
	\ba
		H_\xi (Z) & = & ( \mr{id} _n \otimes \xi ) \left( ( I_n \otimes I _{\mathbb{H} ^2 _d} +Z\otimes L^* ) ( I_n \otimes I_{\mathbb{H} ^2 _d}  - Z \otimes L ^* ) ^{-1} \right) \nn \\
		& = & 2 \sum _{\alpha} Z^\alpha \xi \left( L ^{\alpha ^\mrt}  \right) ^* -I_n \nn \\
		& = & 2 \sum _{k=0} ^\infty Z_1 ^k  - I_n = 2 (I - Z_1 ) ^{-1} - I_n \nn \\
		& = & (I + Z_1) (I -Z_1 ) ^{-1}. \nn
	\ea 
It follows that the Cayley transform, 
$$ b_\xi (Z) := (H_\xi (Z) -I )(H_\xi (Z) + I ) ^{-1}, $$ of $H_\xi$ is $b_\xi (Z) = Z_1$, which is inner.
By \cite[Corollary 7.25]{JM-ncFatou}, $\Pi _\xi$ is the direct sum of a dilation--type and a von Neumann--type row isometry and the previous claim shows that the von Neumann part vanishes.
\end{proof}

\begin{claim}
	For any word $\alpha$ such that $2\in\alpha$, the vector $L^\alpha + N_\xi$ is weak$-*$ continuous.
	In particular, the closed span of $\{L^\alpha+N_\xi:2\in\alpha\}$ is contained in $\rm{WC}(\Pi_\xi)$.
\end{claim}
The proof below uses the concept of the NC Herglotz space of NC Cauchy transforms with respect to a positive NC measure, see \cite[Section 3.8, Lemma 5.2]{JM-ncFatou}. The NC Herglotz space, $\scr{H} ^+ ( H_\mu )$ of any positive NC measure, $\mu \in \posncm$, is a non-commutative reproducing kernel Hilbert space (NC-RKHS) of NC functions in the NC unit row--ball \cite{JM-freeCE,JM-freeAC,JM-ncFatou,JM-ncld}. The details of this construction will not be relevant or needed for our purposes here. It will suffice to remark that if $\mu \in \posncm$ is a positive NC measure, then there is an onto and isometric linear map, $\scr{C} _\mu : \hardy (\mu ) \rightarrow \scr{H} ^+ (H _\mu )$, the \emph{free Cauchy transform}.
\begin{proof}
	Given any $\beta \in \F ^2$ so that $2 \in \beta$, the vector $L^\beta + N _\xi$ is a WC vector if and only if 
	$$ \xi _\beta (L^\alpha) := \ip{L^\beta + N_\xi}{\Pi _\xi^\alpha \,(L^\beta + N_\xi)}_\xi = \xi \left( L^{\beta *} L^\alpha L^\beta \right), $$
	is an absolutely continuous (weak$-*$ continuous) NC measure.
	The free Cauchy transform of $I + N_{\xi _\beta} \in \mathbb{H}^2_2 (\xi _\beta )$, is then,
	$$ h (Z) :=  \sum _{\alpha \in \F ^2} Z^\alpha \xi \left( L^{\beta *} L^{\alpha *} L^\beta \right). $$
	The Taylor coefficients $(h_\alpha)_\alpha$ of $h$ vanish if $\alpha, \beta$ are not comparable.
	Thus, $h_\alpha \neq 0$ if and only if $\alpha = \beta \ga$ or if $\beta = \alpha \ga$.
	Since $\beta \in \F ^2$ is fixed, there are only finitely many words $\alpha$ such that $\beta = \alpha \ga$.
	On the other hand if $\alpha = \beta \ga$, then 
	$$ h_\alpha = h_{\beta \ga} =  \xi  \left( L^{\beta *} L^{\ga *} \right) = 0, $$
	since $2 \in \beta$.
	This proves that $h$ has at most finitely many non-zero Taylor coefficients and so $h \in \C \{ \mf{z} _1, \mf{z} _2 \} \subseteq \mathbb{H} ^2 _2$ and $I + N_{\xi _\beta}$ is a weak$-*$ continuous vector for $\Pi_{\xi_\beta}$. Indeed, since $\scr{C} _{\xi _\beta} (I + N _{\xi _\beta} ) \in \mathbb{H} ^2 _2$, $I + N_{\xi _\beta}$ is a \emph{weak$-*$ analytic vector} for $\xi _\beta$ in the sense of \cite[Definition 8.2]{JM-ncld} and is hence a weak$-*$ continuous vector by \cite[Corollary 8.3]{JM-ncld}.
	
	Since this vector is cyclic and $\rm{WC}(\xi _\beta )$ must be $\Pi_{\xi_\beta}-$invariant,
	we see that $\rm{WC}(\Pi_{\xi_\beta} ) = \mathbb{H} ^2 _2 (\xi _\beta)$.
	Therefore, $\xi _\beta$ is weak$-*$ continuous and $L^\beta + N_\xi$ is a weak$-*$ continuous vector for $\Pi _\xi$.
\end{proof}

\begin{prop} \label{countereg}
	The positive NC measure defined by $\Xi(L^\alpha ) := \xi \left( L_2 ^* \,  L^\alpha \, L_2 \right)$ is equal to NC Lebesgue measure, $m$, and hence is weak$-*$ continuous.
	The NC measure $\ga :=  \xi ^{(2) *} \in \scr{A} _2 ^\dag$, i.e.  $c\mapsto \xi \left( c L_2  \right)$, annihilates the NC disk algebra $\mathbb{A} _2$, but is not weak$-*$ continuous.
\end{prop}
\begin{proof}
	The vector $L_2 + N_\xi$ is a unit wandering vector for the dilation--type positive NC measure $\xi$.
	Hence, $\Xi(L^\alpha ) = \delta _{\alpha , \varnothing } = m (L^\alpha )$.
	Consider the sequence $(L_1 ^k ) ^* L_2 ^* \in \scr{A} _2$.
	This converges weak$-*$ to $0$, and yet, 
	\[ 
		\ga \left( L_1 ^{k *} L_2 ^* \right) = \ov{\xi ( L_1 ^k )} = 1,
	\]
	which of course cannot converge to $0$.
\end{proof}

\bibliography{ncFnM}

\begin{thebibliography}{10}

\bibitem{BMV}
J.A. Ball, G.~Marx, and V.~Vinnikov.
\newblock Noncommutative reproducing kernel {H}ilbert spaces.
\newblock {\em J. Funct. Anal.}, 271:1844--1920, 2016.

\bibitem{BH-Toep}
H.~Arlen Brown and P.~B. Halmos.
\newblock Algebraic properties of toeplitz operators.
\newblock {\em rem}, 100:173, 1963.

\bibitem{CMT-analytic}
Rapha{\"e}l Clou{\^a}tre, Robert~TW Martin, and Edward~J Timko.
\newblock Analytic functionals for the non-commutative disc algebra.
\newblock {\em Accepted by J. Funct. Anal. arXiv:2104.02130}, 2021.

\bibitem{Cuntz}
J.~Cuntz.
\newblock Simple {$C^*-$}algebras generated by isometries.
\newblock {\em Communications in mathematical physics}, 57:173--185, 1977.

\bibitem{KRD-semi}
Kenneth~R. Davidson.
\newblock Free semigroup algebras, a survey.
\newblock In {\em Systems, approximation, singular integral operators, and
  related topics}, pages 209--240. 2001.

\bibitem{DKS-finrow}
Kenneth~R. Davidson, David~W. Kribs, and Miron~E. Shpigel.
\newblock Isometric dilations of non-commuting finite rank n-tuples.
\newblock {\em Canadian Journal of Mathematics}, 53:506--545, 2001.

\bibitem{DP-inv}
Kenneth~R. Davidson and David~R. Pitts.
\newblock Invariant subspaces and hyper-reflexivity for free semigroup
  algebras.
\newblock {\em Proceedings of the London Mathematical Society}, 78:401--430,
  1999.

\bibitem{DKP-structure}
K.R. Davidson, E.~Katsoulis, and D.R. Pitts.
\newblock The structure of free semigroup algebras.
\newblock {\em J. Reine Angew. Math.}, 533:99--126, 2001.

\bibitem{DK-dilation}
K.R. Davidson and E.G. Katsoulis.
\newblock Dilation theory, commutant lifting and semicrossed products.
\newblock {\em Doc. Math.}, 16:781--868, 2011.

\bibitem{DLP-ncld}
K.R. Davidson, J.~Li, and D.R. Pitts.
\newblock Absolutely continuous representations and a {K}aplansky density
  theorem for free semigroup algebras.
\newblock {\em J. Funct. Anal.}, 224:160--191, 2005.

\bibitem{Hoff}
K.~Hoffman.
\newblock {\em Banach spaces of analytic functions}.
\newblock Courier Corporation, 2007.

\bibitem{JM-freeAC}
M.T. Jury and R.T.W. Martin.
\newblock Non-commutative {C}lark measures for the {F}ree and {A}belian
  {T}oeplitz algebras.
\newblock {\em J. Math. Anal. Appl.}, 456:1062--1100, 2017.

\bibitem{JM-freeCE}
M.T. Jury and R.T.W. Martin.
\newblock Column-extreme multipliers of the {F}ree {H}ardy space.
\newblock {\em J. Lond. Math. Soc.}, In press, 2019.

\bibitem{JM-ncld}
M.T. Jury and R.T.W. Martin.
\newblock Lebesgue decomposition of non-commutative measures.
\newblock {\em Int. Math. Res. Not.}, In press, 2020.

\bibitem{JM-ncFatou}
M.T. Jury and R.T.W. Martin.
\newblock Fatou's theorem for non-commutative measures.
\newblock arXiv:1907.09590, 2021.

\bibitem{JMS-ncBSO}
M.T. Jury, R.T.W. Martin, and E.~Shamovich.
\newblock Blaschke--singular--outer factorization of free non-commutative
  functions.
\newblock {\em Advances in Mathematics}, 384:article 107720, 2021.

\bibitem{MK-wand}
Matthew Kennedy.
\newblock Wandering vectors and the reflexivity of free semigroup algebras.
\newblock {\em Journal f{\"u}r die reine und angewandte Mathematik (Crelle's
  Journal)}, 2011:47--53, 2011.

\bibitem{MK-rowiso}
Matthew Kennedy.
\newblock The structure of an isometric tuple.
\newblock {\em Proceedings of the London Mathematical Society}, 106:1157--1177,
  2013.

\bibitem{Paulsen}
V.~Paulsen.
\newblock {\em Completely Bounded Maps and Operator Algebras}.
\newblock Cambridge University Press, New York, NY, 2002.

\bibitem{analysisnow}
G.K. Pedersen.
\newblock {\em Analysis now}.
\newblock Springer, 2012.

\bibitem{Pop-dil}
G.~Popescu.
\newblock Isometric dilations for infinite sequences of noncommuting operators.
\newblock {\em Trans. Amer. Math. Soc.}, 316:523--536, 1989.

\bibitem{Pop-NCdisk}
G.~Popescu.
\newblock Non-commutative disc algebras and their representations.
\newblock {\em Proc. Amer. Math. Soc.}, 124:2137--2148, 1996.

\bibitem{Pop-freeholo}
G.~Popescu.
\newblock Free holomorphic functions on the unit ball of {$B (\mathcal{H})
  ^n$}.
\newblock {\em J. Funct. Anal.}, 241:268--333, 2006.

\bibitem{Pop-entropy}
Gelu Popescu.
\newblock {\em Entropy and multivariable interpolation}.
\newblock American Mathematical Soc., 2006.

\bibitem{Riesz}
F.~Riesz and M.~Riesz.
\newblock \"{U}ber die randwerte einer analytischen funktion.
\newblock {\em Quatri\`eime Congris des Math. Scand. Stockholm}, pages 27--44,
  1916.

\bibitem{SSS}
Guy Salomon, Orr Shalit, and Eli Shamovich.
\newblock Algebras of bounded noncommutative analytic functions on subvarieties
  of the noncommutative unit ball.
\newblock {\em Transactions of the American Mathematical Society},
  370:8639--8690, 2018.

\bibitem{Wittstock}
G.~Wittsock.
\newblock Ein operatorwertiger {H}ahn--{B}anach {S}atz.
\newblock {\em J. Funct. Anal.}, 40:127--150, 1981.

\end{thebibliography}

\Addresses

\end{document}